\documentclass[11pt]{article}
\usepackage{epsf,amsmath,amsfonts,amsthm,graphicx,color,amssymb,caption,subcaption,xcolor}
\numberwithin{equation}{section}

\begin{document}

\theoremstyle{plain}
\newtheorem{Lemma}{Lemma}[section]
\newtheorem{Prop}[Lemma]{Proposition}
\newtheorem{Thm}[Lemma]{Theorem}
\newtheorem{Cor}[Lemma]{Corollary}

\theoremstyle{definition}
\newtheorem{Def}[Lemma]{Definition}
\newtheorem{Rk}[Lemma]{Remark}
\newtheorem{Example}[Lemma]{Example}
\newtheorem{Exercise}[Lemma]{Exercise}

\newcommand{\Natural}{\mbox{${\bf N}$}}
\newcommand{\Integer}{\mbox{${\bf Z}$}}
\newcommand{\Real}{\mbox{${\bf R}$}}
\newcommand{\Complex}{\mbox{${\bf C}$}}

\newcommand{\Eps}{\varepsilon}
\newcommand{\Alpha}{A}
\newcommand{\Beta}{B}
\newcommand{\Epsilon}{E}
\newcommand{\Zeta}{Z}
\newcommand{\Eta}{H}
\newcommand{\Iota}{I}
\newcommand{\Kappa}{K}
\newcommand{\Mu}{M}
\newcommand{\Nu}{N}
\newcommand{\omicron}{o}
\newcommand{\Omicron}{O}
\newcommand{\Rho}{R}
\newcommand{\Tau}{T}
\newcommand{\Chi}{X}

\newcommand{\Sfrac}[2]{\mbox{\small$\frac{#1}{#2}$\normalsize}}
\newcommand{\Half}{\Sfrac{1}{2}}
\newcommand{\Vect}[1]{{\bf #1}}
\newcommand{\Grad}[1]{\nabla #1}
\newcommand{\Gradp}[1]{\nabla' #1}
\newcommand{\Gradx}[1]{\nabla_x #1}
\newcommand{\Gradxp}[1]{\nabla_{x'} #1}
\newcommand{\GradAlphap}[1]{\nabla_{\alpha}' #1}
\newcommand{\Div}[1]{\text{div}\left[#1\right]}
\newcommand{\Divp}[1]{\text{div}'\left[#1\right]}
\newcommand{\Divx}[1]{\text{div}_x \left[#1\right]}
\newcommand{\Divxp}[1]{\text{div}_{x'} \left[#1\right]}
\newcommand{\Curl}[1]{\text{curl}[#1]}
\newcommand{\Laplacian}[1]{\Delta #1}
\newcommand{\Laplacianp}[1]{\Delta' #1}
\newcommand{\Laplacianx}[1]{\Delta_x #1}
\newcommand{\Laplacianxp}[1]{\Delta_{x'} #1}
\newcommand{\Biharmonic}[1]{\Delta^2 #1}
\newcommand{\FT}[1]{{\cal F} \left\{ #1 \right\}}
\newcommand{\FTI}[1]{{\cal F}^{-1} \left\{ #1 \right\}}
\newcommand{\Variation}[2]{\delta_{#2} #1}

\newcommand{\Norm}[2]{\left\|#1\right\|_{#2}}
\newcommand{\LeftNorm}[1]{\left\|#1\right.}
\newcommand{\RightNorm}[2]{\left.#1\right\|_{#2}}
\newcommand{\SupNorm}[1]{\left|\left|#1\right|\right|_{L^{\infty}}}
\newcommand{\HolderNorm}[2]{\left|#1\right|_{C^{#2}}}
\newcommand{\SobNorm}[2]{\left\|#1\right\|_{H^{#2}}}

\newcommand{\InnerProd}[2]{\left\langle#1,#2\right\rangle}
\newcommand{\DotProd}[2]{\left\langle#1,#2\right\rangle}
\newcommand{\Abs}[1]{\left|#1\right|}
\newcommand{\Mod}[1]{\left|#1\right|}
\newcommand{\Angle}[1]{\langle #1 \rangle}
\newcommand{\RealPart}[1]{\text{Re\{}#1\text{\}}}
\newcommand{\ImagPart}[1]{\text{Im\{}#1\text{\}}}
\newcommand{\Null}[1]{\mbox{${\cal N}$}(#1)}
\newcommand{\Ran}[1]{\text{ran}(#1)}
\newcommand{\Ker}[1]{\text{ker}(#1)}
\newcommand{\Dim}[1]{\text{dim}(#1)}
\newcommand{\Rank}[1]{\text{rank}(#1)}
\newcommand{\Det}[1]{\mbox{det} #1}
\newcommand{\Span}[1]{\text{span}(#1)}
\newcommand{\sgn}{\text{sgn}}

\newcommand{\sech}{\mbox{$\mathrm{sech}$}}
\newcommand{\csch}{\mbox{$\mathrm{csch}$}}

\newcommand{\Intersect}{\cap}
\newcommand{\Union}{\cup}

\newcommand{\dftl}[1]{\; d#1}
\newcommand{\dV}{\dftl{V}}
\newcommand{\dS}{\dftl{S}}
\newcommand{\dx}{\dftl{x}}
\newcommand{\dy}{\dftl{y}}
\newcommand{\dz}{\dftl{z}}
\newcommand{\ds}{\dftl{s}}
\newcommand{\dt}{\dftl{t}}
\newcommand{\du}{\dftl{u}}
\newcommand{\dsigma}{\dftl{\sigma}}

\newcommand{\BigOh}[1]{\mathcal{O}(#1)}
\newcommand{\LittleOh}[1]{\mathcal{o}(#1)}

\newcommand{\cA}{\mathcal{A}}
\newcommand{\cB}{\mathcal{B}}
\newcommand{\cC}{\mathcal{C}}
\newcommand{\cD}{\mathcal{D}}
\newcommand{\cE}{\mathcal{E}}
\newcommand{\cF}{\mathcal{F}}
\newcommand{\cG}{\mathcal{G}}
\newcommand{\cH}{\mathcal{H}}
\newcommand{\cI}{\mathcal{I}}
\newcommand{\cJ}{\mathcal{J}}
\newcommand{\cK}{\mathcal{K}}
\newcommand{\cL}{\mathcal{L}}
\newcommand{\cM}{\mathcal{M}}
\newcommand{\cN}{\mathcal{N}}
\newcommand{\cO}{\mathcal{O}}
\newcommand{\cP}{\mathcal{P}}
\newcommand{\cQ}{\mathcal{Q}}
\newcommand{\cR}{\mathcal{R}}
\newcommand{\cS}{\mathcal{S}}
\newcommand{\cT}{\mathcal{T}}
\newcommand{\cU}{\mathcal{U}}
\newcommand{\cV}{\mathcal{V}}
\newcommand{\cW}{\mathcal{W}}
\newcommand{\cX}{\mathcal{X}}
\newcommand{\cY}{\mathcal{Y}}
\newcommand{\cZ}{\mathcal{Z}}

\newcommand{\px}{\partial_x}
\newcommand{\py}{\partial_y}
\newcommand{\pz}{\partial_z}
\newcommand{\pt}{\partial_t}

\newcommand{\sumn}{\sum_{n=0}^{\infty}}
\newcommand{\sumno}{\sum_{n=1}^{\infty}}
\newcommand{\summ}{\sum_{m=0}^{\infty}}
\newcommand{\sump}{\sum_{p \in \Gamma'}}
\newcommand{\sumq}{\sum_{q \in \Gamma'}}

\newcommand{\be}{\begin{equation}}
\newcommand{\ee}{\end{equation}}
\newcommand{\bes}{\begin{equation*}}
\newcommand{\ees}{\end{equation*}}
\newcommand{\bse}{\begin{subequations}}
\newcommand{\ese}{\end{subequations}}

\newcommand{\Schrodinger}{Schr\"odinger}
\newcommand{\Holder}{H\"older}
\newcommand{\Calderon}{Calder\'{o}n}
\newcommand{\Pade}{Pad\'{e}}

\newcommand{\Question}[1]{\fbox{ {\bf Q: #1} }}
\newcommand{\Corrected}[1]
{\noindent \rule{\linewidth}{.75mm} \\ {\bf CORRECTED UP TO HERE (#1)} \\ \rule{\linewidth}{.75mm}}
\newcommand{\void}[1]{}

\newcommand{\RevOne}[1]{\textcolor{red}{#1}}
\newcommand{\RevTwo}[1]{\textcolor{blue}{#1}}
\newcommand{\RevThree}[1]{\textcolor{green}{#1}}

\newcommand{\pN}{\partial_N}
\newcommand{\pa}{\partial_{\alpha}}
\newcommand{\Da}{D_{\alpha}}
\newcommand{\Rational}{\mathbf{Q}}
\newcommand{\Grada}[1]{\nabla_{\alpha} #1}
\newcommand{\Diva}[1]{\text{div}_{\alpha} \left[#1\right]}
\newcommand{\Laplaciana}[1]{\Delta_{\alpha} #1}

\newcommand{\tf}{\tilde{f}}
\newcommand{\tg}{\tilde{g}}
\newcommand{\tu}{\tilde{u}}
\newcommand{\tv}{\tilde{v}}
\newcommand{\tvarphi}{\tilde{\varphi}}
\newcommand{\teta}{\tilde{\eta}}
\newcommand{\txi}{\tilde{\xi}}
\newcommand{\tnu}{\tilde{\nu}}
\newcommand{\tG}{\tilde{G}}
\newcommand{\tT}{\tilde{T}}
\newcommand{\tpsi}{\tilde{\psi}}
\newcommand{\tF}{\tilde{F}}
\newcommand{\tJ}{\tilde{J}}
\newcommand{\tj}{\tilde{j}}
\newcommand{\tL}{\tilde{L}}
\renewcommand{\th}{\tilde{h}}

\newcommand{\tFa}{\tF^{\alpha}}
\newcommand{\hFa}{\hat{F}^{\alpha}}

\newcommand{\AKp}{\Abs{K^T p}}
\newcommand{\AKD}{\Abs{K^T D}}

\newcommand{\Ap}{\Abs{p}}
\newcommand{\AD}{\Abs{D}}
\newcommand{\exct}{{\text{exact}}}
\newcommand{\XNorm}[2]{\left\|#1\right\|_{X^{#2}}}

\newcommand{\blue}[1]{\textcolor{blue}{#1}}
\newcommand{\dave}[1]{\textcolor{red}{#1}}

%
%

\title{Analyticity and Stable Computation of Dirichlet--Neumann
  Operators for Laplace's Equation under Quasiperiodic Boundary
  Conditions in Two and Three Dimensions}
\author{
  David P.\ Nicholls \\
  Department of Mathematics, Statistics, \\
  and Computer Science, \\
  University of Illinois at Chicago, \\
  Chicago, IL 60607, USA.
  \and
  Jon Wilkening \\
  Department of Mathematics, \\
  University of California, \\
  Berkeley, CA 94720, USA.
  \and
  Xinyu Zhao \\
  Department of Mathematical Sciences, \\
  New Jersey Institute of Technology, \\
  Newark, NJ 07102, USA.
  }

\maketitle

\begin{abstract}
Dirichlet--Neumann Operators (DNOs) are important to the formulation,
analysis, and simulation of many crucial models found in engineering
and the sciences. For instance, these operators permit moving--boundary
problems, such as the classical water wave problem (free--surface
ideal fluid flow under the influence of gravity and capillarity), to be 
restated in terms of interfacial quantities, which not only eliminates
the boundary
tracking problem, but also reduces the problem dimension. While these
DNOs have been the object of much recent study regarding their numerical
simulation and rigorous analysis, they have yet to be examined in the
setting of laterally quasiperiodic boundary conditions. The purpose of this
contribution is to begin this investigation with a particular eye
towards the problem of more realistically simulating two and three 
dimensional surface water waves.
Here we not only carefully define the DNO with
respect to these boundary conditions for Laplace's equation, but we
also show the rigorous analyticity of these operators with respect
to sufficiently smooth boundary perturbations. These theoretical 
developments suggest a novel algorithm
for the stable and high--order simulation of the DNO, which we implement
and extensively test.
\end{abstract}

%
%

\section{Introduction}
\label{Sec:Intro}

Many models arising in engineering and science couple partial differential 
equations (PDEs) for volumetric field quantities to their Dirichlet and Neumann
traces at domain boundaries and interfaces. Often, the governing PDEs in the
bulk are rather trivial (e.g., linear and constant coefficient) and the 
simulation challenges (numerical and analytical) arise from the complicated
nature of the problem geometry (e.g., non--separable or unbounded) or 
from nonlinear conditions at the interfaces. For these problems, not only 
can these technical challenges be effectively addressed, but also the 
dimension of the problem can be reduced by restating them entirely in terms
of interfacial quantities (the Dirichlet data). However, this requires 
that the
corresponding surface normal derivatives (the Neumann data) can be
recovered by solving the underlying PDE. Dirichlet--Neumann 
operators (DNOs) accomplish this procedure.  These DNOs arise in a 
wide array of applications ranging from 
linear acoustic \cite{Ihlenburg98},
electromagnetic \cite{Jackson75,Petit80}, 
and
elastic \cite{Achenbach73}
scattering to 
solid \cite{Godreche92} and 
fluid \cite{Lamb,Acheson90}
mechanics. Of particular interest in this contribution are the DNOs
which arise in the study of the water waves problem 
\cite{Lamb,Whitham,CraigSulem93} which
models the free--surface evolution of an ideal fluid under the influence
of gravity and capillarity. In light of these considerations it is clearly
desirable to have a thorough understanding of not only the analytical
properties of these DNOs, but also accurate and stable methods for their
numerical simulation.

All of the classical methods for the numerical approximation of solutions
of PDEs have been brought to bear upon this problem: 
Finite Difference Methods \cite{Strikwerda04,LeVeque07}, 
Finite Element Methods \cite{Johnson87,Gockenbach06},
Spectral Methods \cite{GottliebOrszag77,Boyd01,ShenTangWang11},
and their various generalizations and refinements. However, for the
problems that we have in mind, which feature \textit{homogeneous} PDEs
in the bulk, these \textit{volumetric} solvers are needlessly expensive
as they discretize the full domain rather than simply the boundaries.
Consequently, interfacial methods such as Boundary Integral/Element
Methods (BIM/BEM) have very successfully been applied to the problem of 
simulating DNOs
\cite{BakerMeironOrszag82,HouLowengrubShelley94,Beale01a,WilkeningVasan15}.
Of particular note in the two--dimensional setting, are the conformal mapping
techniques \cite{dyachenko1996analytical, choi1999exact, ruban2004water, dyachenko2016branch, hunter2016two}, which map the fluid domain to the lower half plane or a horizontal strip, thus simplifying the computation of the DNOs to Hilbert transforms which can be easily 
computed in Fourier space. Another major advantage of the conformal mapping techniques is that 
they can be used to compute water waves that are not graphs of functions, such as overhanging waves.
However, these techniques cannot be easily extended to three dimensions.

In this work we advocate for a different class of efficient and accurate
interfacial methods for simulating DNOs which 
view the problem geometry as a deviation from a simpler configuration
(e.g., planar, circular). In summary, the DNO can be shown to depend 
\textit{analytically} upon the deviation magnitude, $\Eps$,
\cite{NichollsReitich99} implying that it can be expressed as
a convergent Taylor series in $\Eps$. The numerical algorithm then
consists of approximating, say, the first $N$ terms in this series and
performing the (finite) summation \cite{NichollsReitich00a}. Initially
one makes the assumption that the perturbation parameter $\Eps$ is
sufficiently small, however, it can be demonstrated that, under certain
reasonable geometrical constraints, the domain of analyticity includes the
\textit{entire} real axis \cite{NichollsReitich00b} so that (real) perturbations
of any size can be simulated provided one has an appropriate analytic
continuation strategy (e.g., Pad\'e summation \cite{BakerGravesMorris96}).
These ``High--Order Perturbation of Surfaces'' (HOPS) methods have been
shown to be highly efficient, accurate, and stable
\cite{NichollsReitich99,NichollsReitich00a,NichollsReitich00b}
within their domain of applicability,
which is certainly the case for water wave simulations up to the
limitations of the formulation (e.g., overturning and breaking).
Importantly, and in contrast to volumetric methods or BEM/BIM, the
derivation, implementation, and rigorous analysis of these HOPS
methods do \textit{not} depend strongly on the problem dimension.

Perhaps due to the extreme analytical and numerical difficulties
presented by the strongly nonlinear
interfacial boundary conditions of the water wave free--boundary
problem, the issue of lateral boundary conditions for this
model has been given
secondary consideration over the years. In fact, almost all
progress has been reported for water waves which exhibit either
periodic patterns \cite{stokes1847, levi1925determination, CraigNicholls99}
or decay to a flat state as the
lateral coordinates approach infinity \cite{Rayleigh1876, vandenBroeck:92, milewski:10}.
However, not
only do such conditions prohibit the study of subharmonic 
(of periods longer than the base period) or noncommensurate 
(of periods irrationally related to the base period)
wave interactions, but also, ocean waves simply look neither periodic
nor decaying at infinity. For these reasons, researchers have studied 
model equations exhibiting quasi-periodic patterning 
\cite{bridges1996spatially, 
WilkeningZhao21a,WilkeningZhao21b,WilkeningZhao23a,WilkeningZhao23b, Sun:Topalov:2023,
Dyachenko:Semenova:2023}. Particularly, two of the authors studied the full water
wave problem with solutions exhibiting quasiperiodic 
patterning in the sense advocated by Moser \cite{Moser66} and made
precise in \S~\ref{Sec:QP}. They
utilized a surface integral formulation specific to water waves
which they addressed with a conformal mapping technique that delivered
highly accurate solutions in a rapid and stable fashion. However, this
scheme is inherently limited to two--dimensional configurations
(one lateral and one vertical) and a three--dimensional version is
highly desirable.

As mentioned in the work of Wilkening \& Zhao 
\cite{WilkeningZhao21b}
one possible
approach to extending their results to higher dimensions is to follow
the work of Craig \& Sulem \cite{CraigSulem93} and restate
the water wave equations in terms of the DNO associated to Laplace's
equation subject to quasiperiodic boundary conditions.
For this one requires not only a fast, stable, and accurate numerical
algorithm for the simulation of these DNOs, but also a rigorous 
analysis justifying their convergence. In this publication we provide,
for the first time in the setting of quasiperiodic boundary conditions,
\textit{three} such HOPS algorithms, together with an analyticity theory
which provides the crucial first step in a full numerical analysis.

The paper is organized as follows: We begin by stating the governing
equations for the water wave problem \cite{Lamb} (\S~\ref{Sec:WWP})
which is the motivation for our study, with a particular discussion
of the quasiperiodic boundary conditions we employ (\S~\ref{Sec:QP}).
We then define the Dirichet--Neumann Operator (DNO), which allows
us to restate the water wave problem in terms of surface variables,
(\S~\ref{Sec:DNO}) with a detailed specification of the transparent
boundary conditions which permit a precise and uniform statement
of the boundary conditions in the far field (\S~\ref{Sec:TBC}).
With this we establish analyticity of the DNO in \S~\ref{Sec:Anal}
which requires a change of variables (\S~\ref{Sec:COV}),
a discussion of function spaces (\S~\ref{Sec:FcnSpaces}),
an elliptic estimate (\S~\ref{Sec:EllEst}), and an inductive
lemma (\S~\ref{Sec:Recur}) to deliver the analyticity results
for the field (\S~\ref{Sec:Anal:Field}) and the DNO
(\S~\ref{Sec:Anal:DNO}).
In \S~\ref{Sec:NumAlg} we describe our numerical algorithms
with a brief summary of our High--Order Spectral approach
in \S~\ref{Sec:HOS}. For the simulation of DNOs subject to
quasiperiodic boundary conditions, we describe our novel 
generalizations of the methods of
Operator Expansions (OE), \S~\ref{Sec:OE},
Field Expansions (FE), \S~\ref{Sec:FE},
Transformed Field Expansions (TFE), \S~\ref{Sec:TFE}.
In \S~\ref{Sec:Pade} we recall the method of Pad\'e approximation
which we use to implement numerical analytic continuation.
Finally, in \S~\ref{Sec:NumRes} we present numerical results of
our implementations of the OE, FE, and TFE algorithms as compared
with exact solutions from the Method of Manufactured Solutions
(\S~\ref{Sec:MMS}). In Appendix~\ref{Sec:EllEst:Proof} we present
a proof of the requisite elliptic estimate under quasiperiodic
boundary conditions which is central to establishing our result.

%
%

\section{The Water Wave Problem}
\label{Sec:WWP}

The classical water wave model simulates the free--surface
evolution of an irrotational, inviscid, and incompressible
(ideal) fluid under the influence of gravity and capillarity
\cite{Lamb,Acheson90}. These Euler equations are posed on
the moving domain
\bes
S_{h,\eta} = \left\{ (x,y) \in \Real^{n-1} \times \Real\ |\
  -h < y < \eta(x,t) \right\},
  \quad
  n \in \{ 2, 3 \},
\ees
where $0 < h \leq \infty$. The irrotational nature of the flow
demands that the fluid velocity be the gradient of a potential,
$\vec{v} = \Grad{\varphi}$, while incompressibility enforces that
the fluid's velocity be divergence free, $\Div{\vec{v}} = 0$.
Therefore, the velocity potential $\varphi$ is harmonic,
\bes
\Laplacianx{\varphi} + \py^2 \varphi = 0.
\ees
In the case of finite depth, the bottom is assumed to be impermeable
so that
\bes
\py \varphi(x,-h) = 0,
\ees
while a fluid of infinite depth mandates
\bes
\py \varphi \rightarrow 0, \qquad y \rightarrow -\infty.
\ees
These are supplemented with initial conditions, and 
the kinematic and Bernoulli conditions at the free surface,
\bse
\label{Eqn:WWP}
\begin{align}
& \pt \eta + \Gradx{\eta} \cdot \Gradx{\varphi} - \py \varphi = 0,
  && y = \eta, \\
& \pt \varphi + \frac{1}{2} \Abs{\Grad{\varphi}}^2 + g \eta
  - \sigma \Divx{ \frac{\Gradx{\eta}}{1 + \Abs{\Gradx{\eta}}^2} }
  = 0,
  && y = \eta.
\end{align}
\ese
All that remains is to specify the lateral boundary conditions which
the velocity potential, $\varphi$, and interface, $\eta$, must satisfy.
For this we choose quasiperiodicity in a sense we now make precise.

%
%

\subsection{Quasiperiodic Functions}
\label{Sec:QP}

There are many notions of quasiperiodicity that have been
advanced in the literature and we focus on the one prescribed
by J.\ Moser \cite{Moser66}. We define a function $f(x)$ to be 
quasiperiodic if
\bes
f(x) = \tf(\alpha),
\quad
\alpha = K x,
\quad
x \in \Real^n,
\quad
\alpha \in \Real^d,
\quad
K \in \Real^{d \times n},
\ees
and the envelope function $\tf(\alpha)$ is \textit{periodic}
with respect to the lattice
\bes
\Gamma = (2 \pi \Integer)^d,
\ees
so that
\bes
\tf(\alpha) = \sum_{p \in \Gamma'} \hat{f}_p e^{i p \cdot \alpha},
\quad
\hat{f}_p = \frac{1}{(2 \pi)^d} \int_{P(\Gamma)} \tf(\alpha) e^{-i p \cdot \alpha}
  \; d \alpha,
\ees
where
\bes
\Gamma' = \Integer^d, \qquad P(\Gamma) = [0, 2\pi)^d.
\ees
In order to specify truly non--periodic
functions we demand
that $d > n$ and the rows of $K$ are linearly
independent over the integers. For example,
in the case $n=1$ and $d=2$, we can choose
\bes
K = \begin{pmatrix} 1 \\ \kappa \end{pmatrix},
\quad
\kappa \not \in \Rational.
\ees
For later use we point out that simple calculations reveal
\bes
\Gradx{f(x)} = K^T \Grada{\tf(\alpha)},
\quad
\Laplacianx{f(x)} = \Diva{ K K^T \Grada{ \tf(\alpha) } }, 
\quad 
\alpha = K x.
\ees
Inspired by this definition of quasiperiodic functions of the
lateral variable $x$, we extend this notion to functions which
are laterally quasiperiodic but not vertically. For instance,
$u$ is said to be laterally quasiperiodic if
\bes
u(x,y) = \tu(\alpha,y),
\quad
\alpha = K x,
\quad
x \in \Real^n,
\quad
\alpha \in \Real^d,
\quad
K \in \Real^{d \times n},
\ees
and $\tu(\alpha,y)$ is $\alpha$--\textit{periodic} with respect
to the lattice $\Gamma = (2 \pi \Integer)^d$ so that
\bes
\tu(\alpha,y) = \sum_{p \in \Gamma'} \hat{u}_p(y) e^{i p \cdot \alpha},
\quad
p \in \Gamma' = \Integer^d.
\ees

%
%

\section{The Dirichlet--Neumann Operator}
\label{Sec:DNO}

Following the work of Craig \& Sulem \cite{CraigSulem93} we
restate the water wave problem in terms of the DNO. Due to the
time--independent nature of the DNO we suppress time dependence
of the free interface in its definition and use the notation
$y = g(x)$ to denote its parameterization, which
is assumed to be quasi--periodic throughout the current work.
For this we require the following definition.
\begin{Def}
Given Dirichlet data, $\xi(x)$, the unique laterally quasiperiodic
solution of the boundary value problem
\bse
\label{Eqn:Laplace:Orig}
\begin{align}
& \Laplacianx{\varphi} + \py^2 \varphi = 0 
  && \text{in $S_{h,g}$}, \\
& \varphi(x,g(x)) = \xi(x), && \text{at $y=g(x)$},
\end{align}
subject to the condition,
\begin{align}
& \py \varphi(x,-h) = 0, && \text{if $h < \infty$}, \\
& \py \varphi \rightarrow 0, \qquad y \rightarrow -\infty,
  && \text{if $h = \infty$},
\end{align}
\ese
specifies the Neumann data,
\bes
\nu(x) = \py \varphi(x,g(x)) 
  - \Gradx{g(x)}  \cdot \Gradx{\varphi(x,g(x))}.
\ees
The Dirichlet--Neumann Operator is defined by
\be
G(g): \xi(x) \rightarrow \nu(x).
\ee
\end{Def}
Given this definition we can restate the water wave problem
as the system of PDEs \cite{CraigSulem93,CraigNicholls99},
\begin{align*}
\pt \eta & = G(\eta) \xi, \\
\pt \xi & = -g \eta - \frac{1}{2 (1 + \Abs{\Gradx{\eta}}^2)}
  \left[ \Abs{\Gradx{\xi}}^2 - (G(\eta) \xi)^2
  - 2 (G(\eta) \xi) \Gradx{\xi} \cdot \Gradx{\eta}
  \right. \\ & \quad \left. 
  + \Abs{\Gradx{\xi}}^2 \Abs{\Gradx{\eta}}^2
  - (\Gradx{\xi} \cdot \Gradx{\eta})^2 \right] + \sigma \Divx{ \frac{\Gradx{\eta}}{1 + \Abs{\Gradx{\eta}}^2} }.
\end{align*}

For our purposes it is more convenient to restate the definition
of the DNO in terms of the independent variable $\alpha$ and
the envelope functions  $\{ \tvarphi, \tg, \txi, \tnu \}$. To do so,
we lift the lower--dimensional quasi--periodic problem in the lateral
direction to a periodic problem defined on a higher--dimensional torus.
Therefore, the equations are defined on the new domain 
\bes
S_{h,\tg} = \left\{ (\alpha,y) \in P(\Gamma) \times \Real\ |\
  -h < y < \tg(\alpha) \right\}.
\ees
\begin{Def}
Given Dirichlet data, $\txi(\alpha)$, the unique laterally periodic
solution of the boundary value problem
\bse
\label{Eqn:Laplace:QP}
\begin{align}
& \Diva{ K K^T \Grada{ \tvarphi(\alpha,y) } } 
  + \py^2 \tvarphi(\alpha,y) = 0,
  && \text{in $S_{h,\tg}$}, \\
& \tvarphi(\alpha,\tg(\alpha)) = \txi(\alpha),
  && y = \tg(\alpha),
\end{align}
subject to the condition,
\begin{align}
& \py \tvarphi(\alpha,-h) = 0, && \text{if $h < \infty$}, 
  \label{Eqn:Laplace:QP:c} \\
& \py \tvarphi \rightarrow 0, \qquad y \rightarrow -\infty,
  && \text{if $h = \infty$},
  \label{Eqn:Laplace:QP:d}
\end{align}
\ese
specifies the Neumann data,
\bes
\tnu(\alpha) = \py \tvarphi(\alpha,\tg(\alpha)) 
  - (K^T \Grada \tg(\alpha)) \cdot
  K^T \Grada \tvarphi(\alpha,\tg(\alpha)).
\ees
The Dirichlet--Neumann Operator is defined by
\bes
\tG(\tg): \txi(\alpha) \rightarrow \tnu(\alpha).
\ees
\end{Def}

%
%

\subsection{Transparent Boundary Conditions}
\label{Sec:TBC}

Regarding boundary conditions at the bottom of the fluid
domain we can state a rigorous formulation for periodic
waves which simultaneously
accounts for a fluid of \textit{any} depth, even infinite.
We begin with the case of a fluid of infinite depth
and select a value $a$ such that $-a < -\SupNorm{\tg}$.
Clearly, beneath the artificial boundary at $y=-a$ the 
bounded quasiperiodic solution of Laplace's equation is
\bes
\tvarphi(\alpha,y) = \sum_{p \in \Gamma'} \hat{\psi}_p 
  e^{\Abs{K^T p} (y+a)} e^{i p \cdot \alpha}.
\ees
From this we have
\bes
\tvarphi(\alpha,-a) = \sum_{p \in \Gamma'} \hat{\psi}_p
  e^{i p \cdot \alpha} =: \tpsi(\alpha).
\ees
Since
\bes
\py \tvarphi(\alpha,y) = \sum_{p \in \Gamma'}
  \hat{\psi}_p \Abs{K^T p} e^{\Abs{K^T p} (y+a)} 
  e^{i p \cdot \alpha},
\ees
we find
\bes
\py \tvarphi(\alpha,-a) = \sum_{p \in \Gamma'} \Abs{K^T p}
  \hat{\psi}_p e^{i p \cdot \alpha}
  =: \Abs{K^T D}[\tpsi(\alpha)],
\ees
which defines the order--one Fourier multiplier
$\Abs{K^T D}$. With this we can state the infinite depth 
transparent boundary condition, \eqref{Eqn:Laplace:QP:d}, as
\bes
\py \tvarphi - \Abs{K^T D}[\tvarphi] = 0,
\quad
y = -a.
\ees

Similarly, in finite depth, $h < \infty$, we choose
$a$ such that $-h < -a < -\SupNorm{\tg}$ so that
beneath the artificial boundary at $y = -a$ the solution
of Laplace's equation satisfying
\bes
\py \tvarphi = 0,
\quad
y = -h,
\ees
is
\bes
\tvarphi(\alpha,y) = \sum_{p \in \Gamma'} \hat{\psi}_p 
  \frac{\cosh(\Abs{K^T p} (h+y))}{\cosh(\Abs{K^T p} (h-a))} 
  e^{i p \cdot \alpha}.
\ees
With this we have
\bes
\tvarphi(\alpha,-a) = \sum_{p \in \Gamma'} \hat{\psi}_p
  e^{i p \cdot \alpha} =: \tpsi(\alpha),
\ees
and, since
\bes
\py \tvarphi(\alpha,y) = \sum_{p \in \Gamma'} \hat{\psi}_p
  \Abs{K^T p} 
  \frac{\sinh(\Abs{K^T p} (h+y))}{\cosh(\Abs{K^T p} (h-a))}
  e^{i p \cdot \alpha},
\ees
we find
\begin{align*}
\py \tvarphi(\alpha,-a) 
  & = \sum_{p \in \Gamma'} \Abs{K^T p}
    \tanh(\Abs{K^T p}(h-a)) \hat{\psi}_p e^{i p \cdot \alpha} \\
  & =: \Abs{K^T D} \tanh((h-a) \Abs{K^T D})[\tpsi(\alpha)],
\end{align*}
which defines the order--one Fourier multiplier
$\Abs{K^T D} \tanh((h-a) \Abs{K^T D})$.
Thus the transparent boundary condition in finite depth,
\eqref{Eqn:Laplace:QP:c}, reads
\bes
\py \tvarphi - \Abs{K^T D} \tanh((h-a) \Abs{K^T D}) [\varphi] = 0,
\quad
y = -a.
\ees

Therefore, if we define
\bes
\tT := \begin{cases} 
  \Abs{K^T D} \tanh((h-a) \Abs{K^T D}), & h < \infty, \\
  \Abs{K^T D}, & h = \infty, \end{cases}
\ees
then we have the uniform statement of the transparent
boundary conditions, \eqref{Eqn:Laplace:QP:c} \& 
\eqref{Eqn:Laplace:QP:d}, as
\be
\label{Eqn:TransBC:QP}
\py \tvarphi - \tT[\tvarphi] = 0,
\quad
y = -a.
\ee
With this we can (equivalently) restate the definition of the DNO.
\begin{Def}
Given Dirichlet data, $\txi(\alpha)$, the unique laterally periodic
solution of the boundary value problem
\bse
\label{Eqn:Laplace:Full:QP}
\begin{align}
& \Diva{ K K^T \Grada{ \tvarphi(\alpha,y) } } 
  + \py^2 \tvarphi(\alpha,y) = 0, 
  && \text{in $S_{h,\tg}$},
  \label{Eqn:Laplace:Full:QP:a} \\
& \tvarphi(\alpha,\tg(\alpha)) = \txi(\alpha),
  && y = \tg(\alpha), 
  \label{Eqn:Laplace:Full:QP:b} \\
& \py \tvarphi - \tT[ \tvarphi ] = 0, && y = -a, 
  \label{Eqn:Laplace:Full:QP:c} \\
& \tvarphi(\alpha+\gamma,y) = \tvarphi(\alpha,y),
  && \gamma \in \Gamma,
  \label{Eqn:Laplace:Full:QP:d}
\end{align}
\ese
specifies the Neumann data,
\be
\tnu(\alpha) = \py \tvarphi(\alpha,\tg(\alpha)) 
  - (K^T \Grada \tg(\alpha)) \cdot
  K^T \Grada \tvarphi(\alpha,\tg(\alpha)).
\ee
The Dirichlet--Neumann Operator is defined by
\be
\label{Eqn:DNO:QP}
\tG(\tg): \txi(\alpha) \rightarrow \tnu(\alpha).
\ee
\end{Def}

%
%

\section{Analyticity of the Dirichlet--Neumann Operator}
\label{Sec:Anal}

We now follow the approach of Nicholls \& Reitich
\cite{NichollsReitich99,NichollsReitich00b} and utilize
the Method of Transformed Field Expansions (TFE) to establish
the analyticity with respect to boundary deformation, $\tg$,
of not only the field, $\tvarphi$, but also the DNO, $\tG$.
The procedure is, by now, well--established and begins with
a domain--flattening change of variables. Once this has been
accomplished, the candidate solution is expanded in a Taylor
series in powers of the interfacial deformation, resulting
in a recurrence of (inhomogeneous) linearized problems to be
solved. These are recursively estimated to establish the
convergence of the Taylor series for the field which then
yields the analyticity of the DNO.

%
%

\subsection{Change of Variables}
\label{Sec:COV}

Consider the change of variables 
(known as the C--Method \cite{CMR80,CDCM82} in 
electromagnetics or $\sigma$--coordinates \cite{Phillips57}
in oceanography)
\bes
\alpha' = \alpha,
\quad
y' = a \left( \frac{y-\tg(\alpha)}{a+\tg(\alpha)}
  \right),
\ees
and the transformed field
\bes
\tu(\alpha',y') = \tvarphi \left( \alpha',
  \left( \frac{a+\tg(\alpha')}{a} \right) y' 
  + \tg(\alpha') \right),
\ees
which, upon dropping the primes for simplicity,
transforms \eqref{Eqn:Laplace:Full:QP} to
\bse
\label{Eqn:Laplace:Full:QP:COV}
\begin{align}
& \Diva{ K K^T \Grada{ \tu(\alpha,y) } } 
  + \py^2 \tu(\alpha,y) = \tF(\alpha,y;\tg,\tu), 
  && -a < y < 0, \\
& \tu(\alpha,0) = \txi(\alpha),
  && y = 0, \\
& \py \tu - \tT[ \tu ] = \tJ(\alpha;\tg,\tu),
  && y = -a, \\
& \tu(\alpha+\gamma,y) = \tu(\alpha,y),
  && \gamma \in \Gamma,
\end{align}
from which one can produce the Neumann data
\be
\tnu(\alpha) = \py \tu(\alpha,0) + \tL(\alpha;\tg, \tu).
\ee
\ese
The forms for $\tF$, $\tJ$, and $\tL$ are readily
derived and are all $\BigOh{\tg}$.

We now make the smallness assumption on $\tg$,
\bes
\tg(\alpha) = \Eps \tf(\alpha),
\quad
\Eps \ll 1,
\quad
\tf = \BigOh{1},
\ees
where previous results indicate that we will be able
to drop the size assumption on $\Eps$ provided that
it is \textit{real} \cite{NichollsReitich00b,NichollsTaber06}.
With this assumption we seek a solution of the form
\be
\tu = \tu(\alpha,y;\Eps) = \sumn \tu_n(\alpha,y) \Eps^n,
\ee
which, upon insertion into \eqref{Eqn:Laplace:Full:QP:COV},
delivers
\bse
\label{Eqn:Laplace:Full:QP:COV:n}
\begin{align}
& \Diva{ K K^T \Grada{ \tu_n(\alpha,y) } } 
  + \py^2 \tu_n(\alpha,y) = \tF_n(\alpha,y),
  && -a < y < 0, \\
& \tu_n(\alpha,0) = \delta_{n,0} \txi(\alpha),
  && y = 0, \\
& \py \tu_n - \tT[ \tu_n ] = \tilde{J}_n(\alpha),
  && y = -a, \\
& \tu_n(\alpha+\gamma,y) = \tu_n(\alpha,y),
  && \gamma \in \Gamma,
\end{align}
where $\delta_{n,m}$ is the Kronecker delta,
from which one can produce the Neumann data
\be
\label{Eqn:Laplace:Full:QP:COV:DNO:n:e}
\tnu_n(\alpha) = \py \tu_n(\alpha,0) + \tL_n(\alpha),
\ee
\ese
which gives $\tG_n(\tf)[\txi] = \tnu_n$. Here we have
\bse
\label{Eqn:Fn}
\be
\tF_n = \Diva{K \tF^{\alpha}_n} + \py \tF^y_n + \tF^0_n,
\label{Eqn:Fn:a}
\ee
where
\begin{align}
a^2 \tF^{\alpha}_n
  & = -2 a \tf (K^T \Grada{\tu_{n-1}})
    + a (a+y) (K^T \Grada{\tf}) \py \tu_{n-1} \notag \\
  & \quad
    - (\tf)^2 (K^T \Grada{\tu_{n-2}})
    + (a+y) \tf (K^T \Grada{\tf}) \py \tu_{n-2},
\label{Eqn:Fn:b}
\end{align}
and
\begin{align}
a^2 \tF^y_n
  & = a(a+y) (K^T \Grada{\tf}) \cdot (K^T \Grada{\tu_{n-1}})
  \notag \\
  & \quad 
    + (a+y) \tf (K^T \Grada{\tf}) \cdot (K^T \Grada{\tu_{n-2}})
    \notag \\
  & \quad 
    - (a+y)^2 \Abs{K^T \Grada{\tf}}^2 \py \tu_{n-2},
\label{Eqn:Fn:c}
\end{align}
and
\begin{align}
a^2 \tF^0_n
  & = a (K^T \Grada{\tf}) \cdot (K^T \Grada{\tu_{n-1}}) 
    + \tf (K^T \Grada{\tf}) \cdot (K^T \Grada{\tu_{n-2}})
    \notag \\
  & \quad
    - (a+y) \Abs{K^T \Grada{\tf}}^2 \py \tu_{n-2},
\label{Eqn:Fn:d}
\end{align}
and 
\be
\label{Eqn:Jn}
a \tJ_n = \tf \tT[\tu_{n-1}(\alpha, -a)],
\ee
and
\begin{align}
\label{Eqn:Ln}
a \tL_n & = - a (K^T \Grada{\tf}) \cdot
    (K^T \Grada{\tu_{n-1}(\alpha,0)})
    - \tf \tnu_{n-1} \notag \\
  & \quad
    - \tf (K^T \Grada{\tf}) \cdot
    (K^T \Grada{\tu_{n-2}(\alpha,0)})
    + a \Abs{K^T \Grada{\tf}}^2 \py \tu_{n-2}(\alpha,0).
\end{align}
\ese

We point out that \eqref{Eqn:Laplace:Full:QP:COV:DNO:n:e} and \eqref{Eqn:Ln}
give the following formula for the $n$--th correction of the DNO (Neumann data)
\begin{align}
\tG_n(\tf)[\txi] & = \py \tu_n(\alpha,0)
  - (K^T \Grada{\tf}) \cdot (K^T \Grada{\tu_{n-1}(\alpha,0)})
  - \frac{1}{a} \tf \tG_{n-1}(\tf)[\txi]
\notag \\ & \quad
  - \frac{1}{a} \tf (K^T \Grada{\tf}) \cdot 
  (K^T \Grada{\tu_{n-2}(\alpha,0)})
  + \Abs{K^T \Grada{\tf}}^2 \py \tu_{n-2}(\alpha,0).
\label{Eqn:tG:COV:n}
\end{align}

%
%

\subsection{Function Spaces}
\label{Sec:FcnSpaces}

We now establish the analyticity of the transformed field, $\tu$, and DNO, $\tG$,
under quasiperiodic boundary conditions. Our proof follows Nicholls \& Reitich
\cite{NichollsReitich99}, with extra attention given to handling the small
divisors and the derivatives along the quasiperiodic direction in the estimates.
We begin by defining the Fourier multipliers $\AKD^q$,
\begin{align*}
& \AKD^q \tilde{\psi}(\alpha) := \sump \AKp^q \hat{\psi}_p e^{i p \cdot \alpha},
&& q \in \Real, \\
& \AKD^q \tilde{w}(\alpha,y) := \sump \AKp^q \hat{w}_p(y) e^{i p \cdot \alpha},
&& q \in \Real.
\end{align*}
Next, we recall the classical interfacial Sobolev spaces
\bes
H^s(P(\Gamma)) = \left\{ \txi(\alpha) \in L^2(P(\Gamma))\ |\ 
  \SobNorm{\txi}{s} < \infty \right\},
\ees
where
\bes
\SobNorm{\txi}{s}^2 
  = \sum_{p \in \Gamma'} \Angle{p}^{2 s} \Abs{\hat{\xi}_p}^2
  = \Norm{\Angle{D}^s \txi}{L^2(P(\Gamma))}^2,
\qquad 
\Angle{p}^s = (1+|p|^2)^{s/2}.
\ees
and volumetric spaces
\bes
X^s(\Omega) = \left\{ \tu(\alpha,y) \in L^2(\Omega)\ |\
  \XNorm{\tu}{s} < \infty \right\}, 
\ees
on
\bes
\Omega := P(\Gamma) \times (-a,0),
\ees
where
\bes
\XNorm{\tu}{s}^2 
  = \sum_{p \in \Gamma'} \Angle{p}^{2s}
  \Norm{\hat{u}_p(y)}{L^2((-a,0))}^2
  = \Norm{\Angle{D}^s \tu}{L^2(\Omega)}^2.
\ees

\begin{Rk}
\label{Rk:AKD:Mapping}
With these definitions we can illustrate a curious, but crucial,
property of the operators $\AKD^q$. If we consider the space of zero--mean
functions
\bes
H_0^s(P(\Gamma)) = \left\{ \txi(\alpha) \in H^s(P(\Gamma))\ |\ 
  \hat{\xi}_0 = 0 \right\},
\ees
then a direct application of the Poincar\'{e} inequality yields
\bes
\label{Eqn:Ap:Estimate}
\Norm{\txi}{s+1} \leq C \Norm{\Grada{\txi}}{s}. 
\ees
However, an estimate of \eqref{Eqn:Ap:Estimate} no longer holds if we replace $\Grada$ by $\AKD$.
Instead one typically settles for a 
``non--resonance'' condition of the form \cite{Moser66}
\be
\label{Eqn:AKp:Estimate}
\gamma \Angle{p}^{-r} < \Abs{K^T p},
\quad
\Abs{p} \geq 1,
\ee
for some $\gamma, r > 0$. With this we can only realize
\begin{align*}
\Norm{\txi}{s - r}^2 
  & = \sump \Angle{p}^{2s - 2 r} \Abs{\hat{\xi}_p}^2
    \leq \sump \gamma^{-2 } \Abs{K^T p}^{2 } \Angle{p}^{2s} 
    \Abs{\hat{\xi}_p}^2 \\
  & = \sump \gamma^{-2 } \Angle{p}^{2s} \Abs{\AKp\hat{\xi}_p}^2
    = \gamma^{-2 } \Norm{\AKD\txi}{s}^2,
\end{align*}
so that
\bes
\AKD \txi \in H_0^s
\quad
\implies
\quad
\txi \in H_0^{s-r}.
\ees
In particular, $\AKD \txi \in H^s$ does \textit{not} imply $\xi \in H_0^{s+1}$
as one might expect.
\end{Rk}

With the definitions of $H^s$, $X^s$, and their norms we can state
and prove the following important Algebra Property 
\cite{Folland76,NichollsReitich99,Evans10}.
\begin{Lemma}
\label{Lemma:Algebra:QP}
Given an integer $s > d/2$ there exists a constant $M = M(s)$ such that
all of the following estimates are true.
\bse
\label{Eqn:Algebra:QP}
\begin{itemize}
\item If $\tf \in H^s(P(\Gamma))$ and $\txi \in H^s(P(\Gamma))$ then
  \be
  \label{Eqn:Algebra:QP:a}
  \SobNorm{\tf \txi}{s} \leq M \SobNorm{\tf}{s} \SobNorm{\txi}{s}.
  \ee
\item If $\tf \in H^s(P(\Gamma))$ and $\tu \in X^s(\Omega)$ then
  \be
  \label{Eqn:Algebra:QP:b}
  \XNorm{\tf \tu}{s} \leq M \SobNorm{\tf}{s} \XNorm{\tu}{s}.
  \ee
\item If $\tg \in H^{s+1/2}(P(\Gamma))$, $\tpsi \in H^s(P(\Gamma))$,
  and $\AKD^{1/2} \tpsi \in H^s(P(\Gamma))$ then
  \be
  \label{Eqn:Algebra:QP:c}
  \SobNorm{\AKD^{1/2}[\tg \tpsi]}{s} \leq M \SobNorm{\tg}{s+1/2}
    \left\{ \SobNorm{\tpsi}{s} + \SobNorm{\AKD^{1/2} \tpsi}{s} \right\}.
  \ee
\item If $\tg \in H^{s+1/2}(P(\Gamma))$, $\tv \in X^s(\Omega)$,
  and $\AKD^{1/2} \tv \in X^s(\Omega)$ then
  \be
  \label{Eqn:Algebra:QP:d}
  \XNorm{\AKD^{1/2}[\tg \tv]}{s} \leq M \SobNorm{\tg}{s+1/2}
    \left\{ \XNorm{\tv}{s} + \XNorm{\AKD^{1/2} \tv}{s} \right\}.
  \ee
\end{itemize}
\ese
\end{Lemma}
\begin{proof}
The proofs of \eqref{Eqn:Algebra:QP:a} and \eqref{Eqn:Algebra:QP:b}
are standard \cite{Evans10}, while the proof of \eqref{Eqn:Algebra:QP:d}
is similar to that of \eqref{Eqn:Algebra:QP:c}. Therefore we now work to
establish \eqref{Eqn:Algebra:QP:c} and begin with
\begin{align*}
\SobNorm{\AKD^{1/2} [\tg \tpsi]}{s}^2 
  & \leq \sum_{q \in \Gamma'} \Angle{q}^{2s} \Abs{K^T q}
    \left\{ \sum_{p \in\Gamma'} \Abs{\hat{g}_{q-p}} 
    \Abs{\hat{\psi}_p} \right\}^2 \\
  & = \sum_{q \in \Gamma'} 
    \left\{ \sum_{p \in \Gamma'} \Angle{q}^s \Abs{K^T q}^{1/2}
    \Abs{\hat{g}_{q-p}} \Abs{\hat{\psi}_p} \right\}^2
    \\
  & \leq C \sum_{q \in \Gamma'} \left\{ \sum_{p \in \Gamma'}  
    \left( \Angle{q-p}^s + \Angle{p}^s \right)
\right. \\ & \left. \quad
    \left( \Abs{K^T (q-p)}^{1/2} + \Abs{K^T p}^{1/2} \right)
    \Abs{\hat{g}_{q-p}} \Abs{\hat{\psi}_p} \right\}^2 \\
  & \leq C C' \sum_{q \in \Gamma'} \left\{ \sum_{p \in \Gamma'}  
    \left( \Angle{q-p}^s + \Angle{p}^s  \right)
\right. \\ & \quad \left.
    \left( \Angle{q-p}^{1/2} + \AKp^{1/2} \right) \Abs{\hat{g}_{q-p}}
    \Abs{\hat{\psi}_p} \right\}^2.
\end{align*}
Using the Young's convolution inequality
\bes
\Norm{f_1*f_2}{\ell^2} \leq \Norm{f_1}{\ell^1} \Norm{f_2}{\ell^2},
\ees
we can show that
\bes
\left\{ \sum_{q \in \Gamma'} \left( \sum_{p \in \Gamma'}  
    \Angle{q-p}^{s+1/2} \Abs{\hat{g}_{q-p}}
    \Abs{\hat{\psi}_p} \right)^2 \right\}^{1/2}
  \leq \SobNorm{\tg}{s+1/2} \sum_{p\in\Gamma'}\Abs{\hat{\psi_p}}
  \leq C \SobNorm{\tg}{s+1/2} \SobNorm{\tpsi}{s}.
\ees
Similarly, we can also prove that
\begin{gather*}
\left\{ \sum_{q \in \Gamma'} \left( \sum_{p \in \Gamma'}  
  \Angle{q-p}^s \Abs{\hat{g}_{q-p}}
  \Abs{K^T p}^{1/2} \Abs{\hat{\psi}_p} \right)^2 \right\}^{1/2}
  \leq C \SobNorm{\tg}{s} \SobNorm{\Abs{K^T D}^{1/2} \tpsi}{s},
  \\
\left\{ \sum_{q \in \Gamma'} \left( \sum_{p \in \Gamma'}
  \Angle{q-p}^{1/2} \Abs{\hat{g}_{q-p}}
  \Angle{p}^s \Abs{\hat{\psi}_p} \right)^2 \right\}^{1/2}
  \leq C \SobNorm{\tg}{s+1/2} \SobNorm{\tpsi}{s},
  \\
\left\{ \sum_{q \in \Gamma'} \left( \sum_{p \in \Gamma'}
  \Abs{\hat{g}_{q-p}} \Angle{p}^s
  \Abs{K^T p}^{1/2} \Abs{\hat{\psi}_p} \right)^2 \right\}^{1/2}
  \leq \SobNorm{\tg}{s} \SobNorm{\Abs{K^T D}^{1/2} \tpsi}{s}.
\end{gather*}
Therefore we have
\bes
\SobNorm{\AKD^{1/2} [\tg \tpsi]}{s} 
  \leq M \SobNorm{\tg}{s+1/2} \left(
  \SobNorm{\tpsi}{s} + \SobNorm{\AKD^{1/2} \tpsi}{s}
  \right).
\ees
\end{proof}

\begin{Rk}
\label{Rk:EllEst:Explain}
At this point we observe how our current theory will differ
in an important way from the periodic case explored in
Nicholls \& Reitich \cite{NichollsReitich00a}. In the latter
we were able to establish
\bes
\SobNorm{\Abs{D}^{1/2} [\tg \tpsi]}{s} 
  \leq M \SobNorm{\tg}{s+1/2} \left(
  \SobNorm{\tpsi}{s} + \SobNorm{\Abs{D}^{1/2} \tpsi}{s}
  \right)
  \leq 2 M \SobNorm{\tg}{s+1/2} 
  \SobNorm{\Abs{D}^{1/2} \tpsi}{s},
\ees
c.f.\ Remark~\ref{Rk:AKD:Mapping}. By constrast, we \textit{cannot} bound
$\SobNorm{\tpsi}{s}$ by $\SobNorm{\AKD^{1/2} \tpsi}{s}$ 
(see Remark~\ref{Rk:AKD:Mapping})
which will necessitate a more powerful elliptic estimate
(Theorem~\ref{Thm:EllEst}) that controls not only
\bes
\left\{ \AKD^{1/2} \tu, \AKD^{1/2} \py \tu, 
  \AKD^{1/2} K^T \Grada{\tu} \right\}
\ees
but also
\bes
\left\{ \tu, \py \tu, K^T \Grada{\tu} \right\}.
\ees
\end{Rk}

From this it is straightforward to establish the following.
\begin{Cor}
Given an integer $s > d/2$ there exists a constant $M = M(s)$
such that all of the following estimates are true.
\bse
\begin{itemize}
\item If $\tg, \tilde h \in H^{s+1/2}(P(\Gamma))$, 
  $\tpsi \in H^s(P(\Gamma))$, and 
  $\AKD^{1/2} \tpsi \in H^s(P(\Gamma))$ then
  \be
  \label{Eqn:Algebra:QP:e}
  \SobNorm{\AKD^{1/2}[\tg \tilde h \tpsi]}{s} \leq 
    M^2 \SobNorm{\tg}{s+1/2} \SobNorm{\th}{s+1/2}
    \left\{ \SobNorm{\tpsi}{s} + \SobNorm{\AKD^{1/2} \tpsi}{s}
    \right\}.
  \ee
\item If $\tg, \th \in H^{s+1/2}(P(\Gamma))$, 
  $\tv \in X^s(\Omega)$, and $\AKD^{1/2} \tv \in X^s(\Omega)$ then
  \be
  \label{Eqn:Algebra:QP:f}
  \XNorm{\AKD^{1/2}[\tg \th \tv]}{s} \leq M^2 \SobNorm{\tg}{s+1/2}
    \SobNorm{\th}{s+1/2}
    \left\{\XNorm{\tv}{s} + \XNorm{\AKD^{1/2} \tv}{s} \right\}.
  \ee
\end{itemize}
\ese
\end{Cor}

Finally, we state the following, readily proven, result for later use.
\begin{Lemma}
There exists a constant $Y = Y(s)$ such that all of the following are true
\bse
\begin{itemize}
\item If $\tv \in X^s(\Omega)$ then
  \be
  \label{Eqn:Y:a}
  \XNorm{(a+y) \tv}{s} \leq Y(s) \XNorm{\tv}{s}.
  \ee
\item If $\AKD^{1/2} \tv \in X^s(\Omega)$ then
  \be
  \label{Eqn:Y:b}
  \XNorm{\AKD^{1/2} [(a+y) \tv]}{s} \leq Y(s) \XNorm{\AKD^{1/2} \tv}{s}.
  \ee
\end{itemize}
\ese
\end{Lemma}

%
%

\subsection{Elliptic Estimate}
\label{Sec:EllEst}

We now state the elliptic estimate we require concerning the
the generic boundary value problem,
\bse
\label{Eqn:BVP:QP}
\begin{align}
& \Diva{ K K^T \Grada{\tu(\alpha,y)} } + \py^2 \tu(\alpha,y) 
  = \tF(\alpha,y), && -a < y < 0, \\
& \tu(\alpha,0) = \txi(\alpha), && \\
& \py \tu(\alpha,-a) - \tT[\tu(\alpha,-a)] = \tJ(\alpha), &&
\end{align}
where
\be
\tF(\alpha,y) := \Diva{ K \tFa(\alpha,y) } + \py \tF^y(\alpha,y) 
  + \tF^0(\alpha,y).
\ee
\ese
For convenience we define the maximum of a number of quantities
we must control in our estimation.
\begin{Def}
Given an integer $s \geq 0$ we define the following maximum
\begin{align*}
\cM_s[\tu] := \max \left\{ \right.
  &\XNorm{\tu}{s},
  \XNorm{\py \tu}{s},
  \XNorm{K^T \Grada{\tu}}{s},
  \\ 
  &\XNorm{\AKD^{1/2} \tu}{s},
  \XNorm{\AKD^{1/2} \py \tu}{s},
  \XNorm{\AKD^{1/2} K^T \Grada{\tu}}{s},
 \\ 
  &\left. \SobNorm{\py \tu(\alpha,0)}{s},
  \SobNorm{K^T \Grada{\tu}(\alpha,0)}{s},
  \SobNorm{\tT[\tu(\alpha,-a)]}{s} \right\}.
\end{align*}
\end{Def}

\begin{Rk}
As we shall see later in Section~\ref{Sec:Recur}, our analyticity result
does not strictly require
estimates of $\XNorm{\tu}{s}$ and $\XNorm{\AKD^{1/2} \tu}{s}$. However,
as their inclusion requires no additional effort, we include them
here for completeness.
\end{Rk}

The elliptic estimate, proven in Appendix~\ref{Sec:EllEst:Proof},
can now be stated as follows.
\begin{Thm}
\label{Thm:EllEst}
Given an integer $s \geq 0$, provided that
\begin{gather*}
\txi \in H^{s+1}(P(\Gamma)),
\quad
\tJ \in H^s(P(\Gamma)),
\\
\tF^j, \AKD^{1/2} \tF^j \in X^s(\Omega),
\quad
\tF^j(\alpha,0) \in H^s(P(\Gamma)),
\quad
j \in \{ \alpha, y, 0 \},
\end{gather*}
and
\bes
\tF^y(\alpha,-a) = 0,
\ees
there exists a unique solution $\tu \in X^s(\Omega)$ of \eqref{Eqn:BVP:QP}
such that, for some $C_e > 0$,
\begin{multline*}
\cM_s[\tu] \leq C_e \left\{ \SobNorm{\txi}{s+1}
  + \SobNorm{\tJ}{s} 
\right. \\ \left.
  + \sum_{j \in \{ \alpha, y, 0 \}}
  \left( \XNorm{\tF^j}{s}
  + \XNorm{\AKD^{1/2} \tF^j}{s} 
  + \SobNorm{ \tF^j(\alpha,0) }{s} \right)
  \right\}.
\end{multline*}
\end{Thm}

\begin{Rk}
To illustrate the advance we have made in the current contribution,
we point out the difference between the elliptic estimates in the 
periodic and quasiperiodic cases. In the (zero--mean) periodic case
\cite{NichollsReitich00a},
\bes
\Laplacianx{u} + \py^2 u = F
\quad
\implies
\quad
\hat{u}_p(y) = \left[ \py^2 + \Abs{p}^2 \right]^{-1} \hat{F}_p
\quad
\implies
\quad
\XNorm{u}{s+2} \leq C_e \XNorm{F}{s},
\ees
i.e., $u$ is two more orders regular than $F$. However, the 
same is \textit{not} true in the (zero--mean) quasiperiodic case where
\bes
\Diva{ K K^T \Grada{\tu} } + \py^2 u = F
\quad
\implies
\quad
\hat{u}_p(y) = \left[ \py^2 + \AKp^2 \right]^{-1} \hat{F}_p,
\ees
and the factor $\AKp$ can be
arbitrarily close to zero as $\Abs{p} \rightarrow \infty$.
This is why $\tu$ only belongs to $X^s$ instead of $X^{s+2}$ in 
Theorem~\ref{Thm:EllEst}. Moreover, $\AKD \tu \in H^s$ does 
\textit{not} imply that $\tu \in H^{s+1}$ or even $\tu \in H^s$
(see Remark~\ref{Rk:AKD:Mapping}).
Thus both $\XNorm{\py \tu}{s}$ and $\XNorm{\AKD^{1/2} \py \tu}{s}$
appear in the elliptic estimates.
\end{Rk}

\begin{Rk}
We point out that the condition $\tF^y(\alpha,-a)=0$ is satisfied in
equation \eqref{Eqn:Fn} above.
\end{Rk}

%
%

\subsection{A Recursive Lemma}
\label{Sec:Recur}

To prove the analyticity of the field we establish the following
recursive estimate.
\begin{Lemma}
\label{Lemma:Recur}
Given an integer $s > d/2$, if $\tf \in H^{s+3/2}(P(\Gamma))$
and
\bes
\cM_s[\tu_n] \leq K_0 B^n,
\quad
\forall\ n < N,
\ees
for constants $K_0, B > 0$, then we have,
for all $j \in \{ \alpha, y, 0 \}$, the estimate
\begin{multline*}
\max \left\{ \SobNorm{\tJ_N}{s}, 
  \XNorm{\tF^j_N}{s},
  \XNorm{\AKD^{1/2} \tF^j_N}{s},
  \SobNorm{\tF^j_N(\alpha,0)}{s} \right\} \\
  \leq K_1 K_0 \left( \SobNorm{\tf}{s+3/2} B^{N-1}
  + \SobNorm{\tf}{s+3/2}^2 B^{N-2} \right),
\end{multline*}
for a positive constant $K_1 > 0$.
\end{Lemma}
\begin{proof}
Noting that, c.f.\ \eqref{Eqn:Fn:a},
\bes
\tF_N = \Diva{K \tF^{\alpha}_N} + \py \tF^y_N + \tF^0_N,
\ees
we focus on one representative term from each $\tF^j_N$;
the proof for the other terms follows similarly.
To begin we consider, c.f.\ \eqref{Eqn:Fn:b},
\bes
\tF^{\alpha}_N = \cA_N + \ldots,
\quad
\cA_N := \frac{(a+y)}{a^2} \tf (K^T \Grada{\tf}) \py \tu_{N-2},
\ees
and estimate, using \eqref{Eqn:Y:b} and \eqref{Eqn:Algebra:QP:e},
\begin{align*}
\XNorm{\cA_N}{s}
  & = \XNorm{ \frac{(a+y)}{a^2} 
    \tf (K^T \Grada{\tf}) \py \tu_{N-2} }{s} \\
  & \leq \frac{Y M^2}{a^2} \SobNorm{\tf}{s} 
    \SobNorm{K^T \Grada{\tf}}{s} \XNorm{\py \tu_{N-2}}{s} \\
  & \leq K_1 \SobNorm{\tf}{s+3/2}^2 K_0 B^{N-2},
\end{align*}
and
\begin{align*}
\XNorm{\AKD^{1/2} \cA_N}{s}
  & = \XNorm{ \AKD^{1/2} \left[
    \frac{(a+y)}{a^2} 
    \tf (K^T \Grada{\tf}) \py \tu_{N-2} \right] }{s} \\
  & \leq \frac{Y M^2}{a^2} \SobNorm{\tf}{s+1/2} 
    \SobNorm{K^T \Grada{\tf}}{s+1/2} \\
  & \quad \times \left\{ \XNorm{\py \tu_{N-2}}{s} 
    + \XNorm{\AKD^{1/2} \py \tu_{N-2}}{s} \right\} \\
  & \leq K_1 \SobNorm{\tf}{s+3/2}^2 K_0 B^{N-2},
\end{align*}
for $K_1$ chosen appropriately. Also, since
\bes
\cA_N(\alpha,0) = \frac{1}{a} \tf (K^T \Grada{\tf}) 
  \py \tu_{N-2}(\alpha,0),
\ees
we have
\begin{align*}
\SobNorm{\cA_N(\alpha,0)}{s}
  & \leq \SobNorm{\frac{\tf}{a} (K^T \Grada{\tf}) 
  \py \tu_{N-2}(\alpha,0)}{s} \\
  & \leq \frac{M^2}{a} \SobNorm{\tf}{s} 
  \SobNorm{K^T \Grada{\tf}}{s} 
  \SobNorm{\py \tu_{N-2}(\alpha,0)}{s} \\
  & \leq K_1 \SobNorm{f}{s+3/2}^2 K_0 B^{N-2},
\end{align*}
for $K_1$ large enough. Next we recall that,
c.f.\ \eqref{Eqn:Fn:c},
\bes
\tF^y_N = \cB_N + \ldots,
\quad
\cB_N := - \frac{(a+y)^2}{a^2} \Abs{K^T \Grada{\tf}}^2 \py \tu_{N-2},
\ees
and compute, with \eqref{Eqn:Y:b} and \eqref{Eqn:Algebra:QP:e},
\begin{align*}
\XNorm{\cB_N}{s}
  & = \XNorm{\frac{(a+y)^2}{a^2} \Abs{K^T \Grada{\tf}}^2
    \py \tu_{N-2}}{s} \\
  & \leq \frac{Y^2 M^2}{a^2} \SobNorm{K^T \Grada{\tf}}{s}^2
    \XNorm{\py \tu_{N-2}}{s} \\
  & \leq K_1 \SobNorm{\tf}{s+3/2}^2 K_0 B^{N-2},
\end{align*}
and
\begin{align*}
\XNorm{\AKD^{1/2} \cB_N}{s}
  & = \XNorm{\AKD^{1/2} \left[ \frac{(a+y)^2}{a^2} 
  \Abs{K^T \Grada{\tf}}^2
    \py \tu_{N-2} \right]}{s} \\
  & \leq \frac{Y^2 M^2}{a^2} \SobNorm{K^T \Grada{\tf}}{s+1/2}^2 \\
  & \quad \times \left\{ \XNorm{\py \tu_{N-2}}{s} 
    + \XNorm{\AKD^{1/2} \py \tu_{N-2}}{s} \right\} \\
  & \leq K_1 \SobNorm{\tf}{s+3/2}^2 K_0 B^{N-2},
\end{align*}
if $K_1$ is chosen large enough. Also, since
\bes
\cB_N(\alpha,0) = - \Abs{K^T \Grada{\tf}}^2 \py \tu_{N-2}(\alpha,0),
\ees
we have
\begin{align*}
\SobNorm{\cB_N(\alpha,0)}{s}
  & \leq \SobNorm{\Abs{K^T \Grada{\tf}}^2 
  \py \tu_{N-2}(\alpha,0)}{s} \\
  & \leq M^2 \SobNorm{K^T \Grada{\tf}}{s}^2
  \SobNorm{\py \tu_{N-2}(\alpha,0)}{s} \\
  & \leq K_1 \SobNorm{f}{s+3/2}^2 K_0 B^{N-2},
\end{align*}
for $K_1$ large enough. Finally, c.f.\ \eqref{Eqn:Fn:d},
\bes
\tF^0_N = \cC_N + \ldots,
\quad
\cC_N := - \frac{(a+y)}{a^2} \Abs{K^T \Grada{\tf}}^2 \py \tu_{N-2},
\ees
and we estimate, using \eqref{Eqn:Y:b} and \eqref{Eqn:Algebra:QP:e},
\begin{align*}
\XNorm{\cC_N}{s}
  & = \XNorm{\frac{(a+y)}{a^2} \Abs{K^T \Grada{\tf}}^2
    \py \tu_{N-2} }{s} \\
  & \leq \frac{Y M^2}{a^2} \SobNorm{K^T \Grada{\tf}}{s}^2
  \XNorm{\py \tu_{N-2}}{s} \\
  & \leq K_1 \SobNorm{\tf}{s+3/2}^2 K_0 B^{N-2},
\end{align*}
and
\begin{align*}
\XNorm{\AKD^{1/2} \cC_N}{s}
  & = \XNorm{\AKD^{1/2} \left[ \frac{(a+y)}{a^2} 
  \Abs{K^T \Grada{\tf}}^2
    \py \tu_{N-2}\right]}{s} \\
  & \leq \frac{Y M^2}{a^2} \SobNorm{K^T \Grada{\tf}}{s+1/2}^2 \\
  & \quad \times \left\{ 2 \XNorm{\py \tu_{N-2}}{s} 
    + \XNorm{\AKD^{1/2} \py \tu_{N-2}}{s} \right\} \\
  & \leq K_1 \SobNorm{\tf}{s+3/2}^2 K_0 B^{N-2},
\end{align*}
for $K_1$ sufficiently large. Also, since
\bes
\cC_N(\alpha,0) = - \frac{1}{a} \Abs{K^T \Grada{\tf}}^2 \py 
  \tu_{N-2}(\alpha,0),
\ees
we have
\begin{align*}
\SobNorm{\cC_N(\alpha,0)}{s}
  & \leq \SobNorm{\frac{1}{a} 
  \Abs{K^T \Grada{\tf}}^2 \py \tu_{N-2}(\alpha,0)}{s} \\
  & \leq \frac{M^2}{a} \SobNorm{K^T \Grada{\tf}}{s}^2
  \SobNorm{\py \tu_{N-2}(\alpha,0)}{s} \\
  & \leq K_1 \SobNorm{f}{s+3/2}^2 K_0 B^{N-2},
\end{align*}
for $K_1$ large enough. To close, c.f.\ \eqref{Eqn:Jn},
\bes
\tJ_N = \frac{1}{a} \tf \tT[\tu_{N-1}(\alpha, -a)],
\ees
and we compute, using \eqref{Eqn:Algebra:QP:a},
\begin{align*}
\SobNorm{\tJ_N}{s}
  & = \SobNorm{\frac{1}{a} \tf \tT[\tu_{N-1}(\alpha, -a)]}{s}
    \leq \frac{M}{a} \SobNorm{\tf}{s}
    \SobNorm{\tT[\tu_{N-1}(\alpha, -a)]}{s} \\
  & \leq \frac{M}{a} \SobNorm{\tf}{s} K_0 B^{N-1}
    \leq K_1 \SobNorm{\tf}{s+3/2} K_0 B^{N-1},
\end{align*}
for $K_1$ big enough.
\end{proof}

%
%

\subsection{Analyticity of the Field}
\label{Sec:Anal:Field}

We are now in a position to establish the analyticity of the 
transformed field $\tu$ in a sense which we now make precise.
\begin{Thm}
\label{Thm:AnalField}
Given an integer $s > d/2$, if $\tf \in H^{s+3/2}(P(\Gamma))$
and $\txi \in H^{s+1}(P(\Gamma))$ there exists a unique 
solution
\bes
\tu(\alpha,y;\Eps) = \sumn \tu_n(\alpha,y) \Eps^n,
\ees
of \eqref{Eqn:Laplace:Full:QP:COV} satisfying
\be
\label{Eqn:AnalField}
\cM_s[\tu_n] \leq K_0 B^n,
\quad
\forall\ n \geq 0,
\ee
for any $B > C_0 \SobNorm{\tf}{s+3/2}$ and positive constants
$K_0, C_0 > 0$.
\end{Thm}
\begin{proof}
We work by induction on $n$. At order $n=0$ we must solve
\eqref{Eqn:Laplace:Full:QP:COV:n} where
$\tF_0 \equiv \tJ_n \equiv 0$. From Theorem~\ref{Thm:EllEst}
we have that
\bes
\cM_s[\tu_0] \leq C_e \SobNorm{\txi}{s+1} =: K_0 < \infty.
\ees
We now assume estimate \eqref{Eqn:AnalField} for all $n < N$ and
study $\cM_s[\tu_N]$. For this we invoke the elliptic
estimate, Theorem~\ref{Thm:EllEst}, to realize that
\bes
\cM_s[u_N] \leq C_e \left\{ \SobNorm{\tJ_N}{s} 
  + \sum_{j \in \{ \alpha, y, 0 \}}
  \left( \XNorm{\tF^j_N}{s}
  + \XNorm{\AKD^{1/2} \tF^j_N}{s}
  + \SobNorm{ \tF^j_N(\alpha,0) }{s} \right)
  \right\}.
\ees
From the recursive estimate in Lemma~\ref{Lemma:Recur} we now
deduce that
\bes
\cM_s[u_N] \leq C_e 10 K_1 K_0
  \left( \SobNorm{\tf}{s+3/2} B^{N-1}
  + \SobNorm{\tf}{s+3/2} B^{N-2} \right).
\ees
We realize
\bes
\cM_s[u_N] \leq K_0 B^N,
\ees
provided that, for instance,
\bes
B > \max \left\{ 20 C_e K_1, \sqrt{20 C_e K_1} \right\}
  \SobNorm{\tf}{s+3/2},
\ees
and we are done.
\end{proof}

%
%

\subsection{Analyticity of the DNO}
\label{Sec:Anal:DNO}

At last we can establish the analyticity of the DNO, $\tG$,
more specifically we prove the following result.
\begin{Thm}
\label{Thm:DNOAnal}
Given an integer $s > d/2$, if $\tf \in H^{s+3/2}(P(\Gamma))$ and
$\txi \in H^{s+1}(P(\Gamma))$ then the series
\be
\label{Eqn:tG:Exp}
\tG(\Eps \tf) = \sumn \tG_n(\tf) \Eps^n,
\ee
converges strongly as an operator from $H^{s+1}(P(\Gamma))$ to
$H^s(P(\Gamma))$. More precisely,
\be
\label{Eqn:DNOAnal}
\SobNorm{\tG_n(\tf)[\txi]}{s} \leq K_2 B^n,
\quad
\forall\ n \geq 0,
\ee
for any $B > C_0 \SobNorm{\tf}{s+3/2}$ and positive constant
$K_2 > 0$.
\end{Thm}
\begin{proof}
We work by induction in $n$ and begin with the formula for $\tG_0$,
\bes
\tG_0[\txi] = \begin{cases}
  \sump \AKp \hat{\xi}_p e^{i \alpha \cdot p}, & h = \infty, \\
  \sump \AKp \tanh(h \AKD) \hat{\xi}_p e^{i \alpha \cdot p}, & h < \infty.
  \end{cases}
\ees
We focus our attention on the infinite depth case ($h = \infty$) and note that
the finite depth case ($h < \infty$) can be established in a similar fashion.
We estimate
\bes
\SobNorm{\tG_0[\txi]}{s}^2
  \leq \sump \Angle{p}^{2s} \AKp^2 \Abs{\hat{\xi}_p}^2
  \leq C \sump \Angle{p}^{2 (s+1)} \Abs{\hat{\xi}_p}^2
  \leq \SobNorm{\txi}{s+1}^2 =: K_2.
\ees
We now assume \eqref{Eqn:DNOAnal} for all $n < N$ and, from
\eqref{Eqn:tG:COV:n}, we estimate
\begin{align*}
\SobNorm{\tG_N(\tf)[\txi]}{s} 
  & \leq \SobNorm{\py \tu_N(\alpha,0)}{s}
    + \SobNorm{(K^T \Grada{\tf}) \cdot (K^T \Grada{\tu_{N-1}(\alpha,0)})}{s}
\\ & \quad
    + \frac{1}{a} \SobNorm{\tf \tG_{N-1}(\tf)[\txi]}{s}
\\ & \quad
  + \frac{1}{a} \SobNorm{\tf (K^T \Grada{\tf})
  \cdot (K^T \Grada{\tu_{N-2}(\alpha,0)})}{s}
\\ & \quad
  + \SobNorm{(K^T \Grada{\tf})^2 \py \tu_{N-2}(\alpha,0)}{s}.
\end{align*}
Now, from \eqref{Eqn:Algebra:QP:a} we have
\begin{align*}
\SobNorm{\tG_N(\tf)[\txi]}{s} 
  & \leq \SobNorm{\py \tu_N(\alpha,0)}{s}
    + M \SobNorm{K^T \Grada{\tf}}{s} \SobNorm{K^T \Grada{\tu_{N-1}(\alpha,0)}}{s}
\\ & \quad
    + \frac{1}{a} M \SobNorm{\tf}{s} \SobNorm{\tG_{N-1}(\tf)[\txi]}{s}
\\ & \quad
  + \frac{1}{a} M^2 \SobNorm{\tf}{s} \SobNorm{K^T \Grada{\tf}}{s}
  \SobNorm{K^T \Grada{\tu_{n-2}(\alpha,0)}}{s}
\\ & \quad
  + M^2 \SobNorm{K^T \Grada{\tf}}{s}^2 \SobNorm{\py \tu_{N-2}(\alpha,0)}{s}.
\end{align*}
From \eqref{Eqn:AnalField} we have
\begin{align*}
\SobNorm{\tG_N(\tf)[\txi]}{s} 
  & \leq K_0 B^N
    + M \SobNorm{\tf}{s+3/2} K_0 B^{N-1}
    + \frac{1}{a} M \SobNorm{\tf}{s+3/2} K_2 B^{N-1}
\\ & \quad
  + \frac{1}{a} M^2 \SobNorm{\tf}{s+3/2} \SobNorm{\tf}{s+3/2} K_0 B^{N-2}
  + M^2 \SobNorm{\tf}{s+3/2}^2 K_0 B^{N-2}
\\ & \leq K_2 B^N,
\end{align*}
provided that each of the five terms is bounded above by $(K_2/5) B^N$.
This demands that
\bes
K_2 \geq 5 K_0,
\quad
B \geq M \max \left\{ \frac{5 K_0}{K_2}, \frac{5}{a},
  \sqrt{\frac{5 K_0}{a K_2}}, \sqrt{\frac{5 K_0}{K_2}} \right\}
  \SobNorm{\tf}{s+3/2},
\ees
and accommodating all of these delivers the theorem.
\end{proof}

%
%

\section{Numerical Algorithms}
\label{Sec:NumAlg}

In order to illustrate the possibilities enabled by the analyticity theory
outlined above, we now briefly describe and validate three High--Order
Perturbation of Surfaces (HOPS) algorithms 
\cite{NichollsReitich99,NichollsReitich00a,NichollsReitich00b}
suggested by this. Presently we focus on simple two--dimensional 
geometries ($x, y \in \Real$) which feature lateral quasiperiodicity
specified by the matrix $K \in \Real^{2 \times 1}$ (two base periods). 
However, we also describe a limited selection of much more intensive
three--dimensional simulations ($x \in \Real^2$, $y \in \Real$) with
lateral quasiperiodicity determined by $K \in \Real^{3 \times 2}$.

%
%

\subsection{High--Order Spectral Methods}
\label{Sec:HOS}

In light of the separable geometries we consider (the fundamental period
cell determined by the lattice $\Gamma$, $P(\Gamma)$, in tensor product 
with the interval $[-a,0]$) we chose High--Order Spectral (HOS)
methods to discretize each of our algorithms. The interested reader
should consult one of the excellent texts on the topic
(e.g., \cite{GottliebOrszag77,CHQZ88,Boyd01,ShenTang06})
for complete details as we now
provide only a brief description of the concepts we require in this
subsection.

In our schemes we approximated doubly $2 \pi$--periodic functions
by their truncated Fourier series, e.g.,
\bes
\tf(\alpha) \approx \tf^{N_{\alpha}}(\alpha)
  := \sum_{p \in \cP_{N_{\alpha}}} \hat{f}_p e^{i p \cdot \alpha},
\ees
where
\begin{gather*}
p = \begin{pmatrix} p_1 \\ p_2 \end{pmatrix},
\quad
\alpha = \begin{pmatrix} \alpha_1 \\ \alpha_2 \end{pmatrix},
\quad
N_{\alpha} = \begin{pmatrix} N_{\alpha}^1 \\ N_{\alpha}^2 \end{pmatrix},
\\
\cP_{N_{\alpha}} = \left\{ p \in \Gamma'\ \left|\
  -\frac{N_{\alpha}^m}{2} \leq p_m \leq 
  \frac{N_{\alpha}^m}{2} - 1,
  \quad m = 1, 2 \right. \right\}.
\end{gather*}
The gradient of such a function can be readily approximated with
high accuracy by
\bes
\Grada{\tf(\alpha)} \approx 
  \sum_{p \in \cP_{N_{\alpha}}} (i p) \hat{f}_p e^{i p \cdot \alpha},
\ees
as can Fourier multipliers, e.g.,
\bes
m(D)[\tf(\alpha)] \approx 
  \sum_{p \in \cP_{N_{\alpha}}} m(p) \hat{f}_p e^{i p \cdot \alpha}.
\ees
Using well--known procedures involving the Fast Fourier Transform (FFT),
products of these functions can be simulated at the equally--spaced
gridpoints
\bes
\alpha_j^m = \frac{2 \pi j_m}{N_{\alpha}^m},
\quad
0 \leq j_m \leq N_{\alpha}^m-1,
\quad
m = 1, 2.
\ees
This can also be used in conjunction with the Trapezoidal Rule to
approximate the Fourier coefficients, $\hat{f}_p$, with very high
accuracy (in fact, spectral, if $\tf(\alpha)$ is analytic)
\cite{GottliebOrszag77,CHQZ88,Boyd01,ShenTang06}.

In relation to the recursions we have analyzed in our theory above, we also
require the approximation of volumetric functions. For this we
simulate laterally doubly $2 \pi$--periodic functions of $y$ on
the interval $[-a,0]$ by their truncated Fourier--Chebyshev series, e.g.,
\bes
\tilde{u}(\alpha,y) \approx \tilde{u}^{N_{\alpha},N_y}(\alpha,y)
  := \sum_{p \in \cP_{N_{\alpha}}} \sum_{q=0}^{N_y} \hat{u}_{p,q}
  T_q \left( \frac{2y+a}{a} \right) e^{i p \cdot \alpha}.
\ees
Again, classical techniques from the theory of Spectral Methods can
be used to compute derivatives of these functions with respect to 
both $\alpha$ and $y$.
As above, using methods involving the Fast Chebyshev Transform,
products of these functions can be simulated at the Chebyshev points
\bes
y_r = \frac{a}{2} \left( \cos \left( \frac{\pi r}{N_y} \right) - 1 \right),
\quad 
0 \leq r \leq N_y.
\ees
Once again, these can also be used in conjunction with quadrature
formulas to approximate the Fourier--Chebyshev coefficients, 
$\hat{u}_{p,q}$, with high fidelity.

%
%

\subsection{The Method of Operator Expansions}
\label{Sec:OE}

The first HOPS method we discuss
is the classical Method of Operator Expansions (OE) due to Milder
\cite{Milder91a,Milder91b,MilderSharp91,MilderSharp92}
and Craig \& Sulem \cite{CraigSulem93}. In the infinite depth case
($h=\infty$) the description of this approach begins by defining the 
function
\bes
\tvarphi_p(\alpha,y) := e^{\Abs{K^T p} y} e^{i p \cdot \alpha},
\ees
for some $p \in \Integer^2$. This function satisfies the quasiperiodic Laplace's
equation, \eqref{Eqn:Laplace:Full:QP:a}, and the boundary conditions,
\eqref{Eqn:Laplace:Full:QP:c} and \eqref{Eqn:Laplace:Full:QP:d}, so that we can 
insert it into the definition of the quasiperiodic DNO,
\eqref{Eqn:DNO:QP}, to realize the true statement
\bes
\tG(\tg)[ \tvarphi_p(\alpha,\tg) ] = \py \tvarphi_p(\alpha,\tg)
  - \left( K^T \Grada{ \tg } \right)
  \left( K^T \Grada{ \tvarphi(\alpha,\tg) } \right).
\ees
Setting $\tg(\alpha) = \Eps \tf(\alpha)$ and using the analyticity of the
DNO with respect to $\Eps$, we expand to find
\begin{multline*}
\left( \sumn \tG_n(\tf) \Eps^n \right)
  \left[ \summ \tF_m \Abs{K^T p}^m e^{i p \cdot \alpha} \Eps^m \right]
  = \sumn \tF_n \Abs{K^T p}^{n+1} e^{i p \cdot \alpha} \Eps^n \\
  - \left( \Eps K^T \Grada{ \tf } \right)
  \left( \sumn \tF_n (i K^T p) \Abs{K^T p}^n e^{i p \cdot \alpha} 
  \Eps^n \right),
\end{multline*}
where
\bes
\tF_n(\alpha) := \frac{\tf(\alpha)^n}{n!}.
\ees
At order $\BigOh{\Eps^0}$ we find
\bes
\tG_0 \left[ e^{i p \cdot \alpha} \right]
  = \Abs{K^T p} e^{i p \cdot \alpha},
\ees
while at order $\BigOh{\Eps^n}$ for $n>0$ we find
\begin{multline*}
\tG_n(\tf) \left[ e^{i p \cdot \alpha} \right]
  = \tF_n \Abs{K^T p}^{n+1} e^{i p \cdot \alpha}
  - \left( K^T \Grada{\tf} \right) \tF_{n-1} (i K^T p)
  \Abs{K^T p}^{n-1} e^{i p \cdot \alpha}
  \\
  - \sum_{\ell=0}^{n-1} \tG_{\ell}(\tf) \left[ \tF_{n-\ell} 
  \Abs{K^T p}^{n-\ell} e^{i p \cdot \alpha} \right].
\end{multline*}
Now, if we express the generic function $\txi$ in the
Fourier series
\bes
\txi(\alpha) = \sump \hat{\xi}_p e^{i p \cdot \alpha},
\ees
then we can conclude that
\be
\label{Eqn:OE:0}
\tG_0 \left[ \txi \right]
  = \tG_0 \left[ \sump \hat{\xi}_p e^{i p \cdot \alpha} \right]
  = \sump \hat{\xi}_p \tG_0 \left[ e^{i p \cdot \alpha} \right]
  = \sump \Abs{K^T p} \hat{\xi}_p e^{i p \cdot \alpha}
  =: \Abs{K^T D} \txi,
\ee
which we use to define the order--one Fourier multiplier
$\Abs{K^T D}$. In a similar fashion we express, for $n>0$,
\begin{multline*}
\tG_n(\tf) \left[ \txi \right]
  = \tF_n \Abs{K^T D}^{n+1} \txi
  - \left( K^T \Grada{\tf} \right) \tF_{n-1} (i K^T D)
  \Abs{K^T D}^{n-1} \txi
  \\
  - \sum_{\ell=0}^{n-1} \tG_{\ell}(\tf) \left[ 
  \tF_{n-\ell} \Abs{K^T D}^{n-\ell} \txi \right],
\end{multline*}
and we point out the \textit{recursive} nature of this
formula: $\tG_{\ell}$ must be recomputed for all
$0 \leq \ell \leq n-1$ at each perturbation order $n$.

It was noted in \cite{NichollsPhD,Nicholls98} that the
self--adjointness of the DNO could be used to advantage in producing
a \textit{rapid} version of these recursions. Recalling that the
adjoint operation satisfies $(AB)^* = B^* A^*$ for any two linear
operators $A$ and $B$, we can express $\tG_n = \tG_n^*$ as
\begin{multline}
\label{Eqn:OE:n}
\tG_n(\tf) \left[ \txi \right]
  = \Abs{K^T D}^{n+1} \left[ \tF_n \txi \right]
  - \Abs{K^T D}^{n-1} (i K^T D) \left[ \left( K^T \Grada{\tf} \right)
  \tF_{n-1} \txi \right]
  \\
  - \sum_{\ell=0}^{n-1} \Abs{K^T D}^{n-\ell} 
  \left[ \tF_{n-\ell} \tG_{\ell}(\tf)\left[ \txi \right] \right],
\end{multline}
which is much faster than the formula above as the $\tG_{\ell}[\txi]$
may simply be stored from one perturbation order to the next rather than
recomputed from scratch.

The OE algorithm is a Fourier HOS discretization of the formulas 
\eqref{Eqn:OE:0} and \eqref{Eqn:OE:n} where we begin by considering
$\tg(\alpha) = \Eps \tf(\alpha)$, expanding
\bes
\tG(\alpha;\Eps) = \sumn \tG_n \Eps^n,
\ees
and approximating
\bes
\tG(\alpha;\Eps) \approx \tG^N(\alpha;\Eps)
  := \sum_{n=0}^{N} \tG_n(\alpha) \Eps^n,
\ees
where
\be
\label{Eqn:OE:Approx:Nalpha:N}
\tG_n(\alpha) \approx \tG_n^{N_{\alpha}}(\alpha)
  := \sum_{p \in \cP_{N_{\alpha}}} \hat{G}_{p,n} e^{i p \cdot \alpha}.
\ee
Given
\bes
\txi^{N_{\alpha}}(\alpha) = \sum_{p \in \cP_{N_{\alpha}}}
  \hat{\xi}_p e^{i p \cdot \alpha},
\ees
we then insert these forms into \eqref{Eqn:OE:0} and \eqref{Eqn:OE:n}, and 
demand these be true for $0 \leq n \leq N$. Using the approach outlined
in Section~\ref{Sec:HOS} the coefficients $\{ \hat{G}_{p,n} \}$ are readily 
recovered enabling the approximation $\tnu^{N,N_{\alpha}}$.

%
%

\subsection{The Method of Field Expansions}
\label{Sec:FE}

The second HOPS method we discuss is Bruno \& Reitich's Method of Field 
Expansions (FE) \cite{BrunoReitich93a,BrunoReitich93b,BrunoReitich93c}. 
To describe this approach we again consider $\tg(\alpha) = \Eps \tf(\alpha)$
and note that the solution
of \eqref{Eqn:Laplace:Full:QP} will depend upon $\Eps$; we assume that this
dependence is \textit{analytic} so that
\bes
\tvarphi(\alpha,y;\Eps) = \sumn \tvarphi_n(\alpha,y) \Eps^n.
\ees
It can be shown that these $\tvarphi_n$ satisfy
\bse
\label{Eqn:LaplaceFull:n}
\begin{align}
& \Diva{ K K^T \Grada{ \tvarphi_n(\alpha,y) } } 
  + \py^2 \tvarphi_n(\alpha,y) = 0, 
  && -a < y < 0,
  \label{Eqn:LaplaceFull:n:a} \\
& \tvarphi_n(\alpha,0) = \delta_{n,0} \txi(\alpha)
  + Q_n(\alpha),
  && y = 0,
  \label{Eqn:LaplaceFull:n:b} \\
& \py \tvarphi_n - \tT[ \tvarphi_n ] = 0, && y = -a, 
  \label{Eqn:LaplaceFull:n:c} \\
& \tvarphi_n(\alpha+\gamma,y) = \tvarphi_n(\alpha,y),
  && \gamma \in \Gamma.
  \label{Eqn:LaplaceFull:n:d}
\end{align}
where
\be
\label{Eqn:LaplaceFull:Q:n}
Q_n(\alpha) = -\sum_{\ell=0}^{n-1} \tF_{n-\ell} 
  \py^{n-\ell} \tvarphi_{\ell}(\alpha,0).
\ee
\ese
In the infinite depth case ($h=\infty$) the solutions of 
\eqref{Eqn:LaplaceFull:n:a}, \eqref{Eqn:LaplaceFull:n:c}, and 
\eqref{Eqn:LaplaceFull:n:d} can be expressed as
\bes
\tvarphi_n(\alpha,y) = \sump a_{n,p} e^{\Abs{K^T p} y} e^{i p \cdot \alpha},
\ees
while \eqref{Eqn:LaplaceFull:n:b} and \eqref{Eqn:LaplaceFull:Q:n}
can be used to discover the $\{ a_{n,p} \}$. More specifically,
\be
\label{Eqn:FE:n}
a_{n,p} = \delta_{n,0} \hat{\xi}_p 
  - \sumq \sum_{\ell=0}^{n-1} \hat{F}_{n-\ell,p-q}
  \Abs{K^T q}^{n-\ell} a_{\ell,q},
\ee
which constitute the FE recursions.

The FE approach is a Fourier HOS approximation of the recursions
\eqref{Eqn:FE:n} which recovers the $\{ a_{n,p} \}$ given the
Fourier coefficients $\{ \hat{\xi}_p, \hat{f}_p \}$. With these it
is not difficult to approximate the $\tG_n$ which gives a simulation
of the form \eqref{Eqn:OE:Approx:Nalpha:N}.

%
%

\subsection{The Method of Transformed Field Expansions}
\label{Sec:TFE}

The final HOPS scheme we discuss is Nicholls and Reitich's Method 
of Transformed Field Expansions (TFE)
\cite{NichollsReitich99,NichollsReitich00a,NichollsReitich00b,NichollsReitich03a,NichollsReitich03b}.
Simply put, this algorithm is a HOS
\cite{GottliebOrszag77,CHQZ88,Boyd01,ShenTang06}
discretization and approximation
of the transformed problem \eqref{Eqn:Laplace:Full:QP:COV:n} which we 
utilized in our theoretical demonstrations earlier. In a bit more detail
we consider $\tg(\alpha) = \Eps \tf(\alpha)$ and use our \textit{rigorous}
demonstration of the \textit{analyticity} of the solution to write
\bes
\tu(\alpha,y;\Eps) = \sumn \tu_n(\alpha,y) \Eps^n.
\ees
We begin by approximating
\be
\label{Eqn:TFE:Approx:N}
\tu(\alpha;y;\Eps) \approx \tu^N(\alpha,y;\Eps)
  := \sum_{n=0}^{N} \tu_n(\alpha,y) \Eps^n,
\ee
and then represent
\bes
\tu_n(\alpha,y) \approx \sum_{p \in \cP_N} \sum_{q=0}^{N_y}
  \hat{u}_{p,q,n} e^{i p \cdot \alpha} 
  T_q \left( \frac{2z+a}{a} \right).
\ees
We then insert this form into \eqref{Eqn:Laplace:Full:QP:COV:n} and 
utilize the collocation approach as described in Section~\ref{Sec:HOS}.
Using well--known procedures involving the FFT and
the Fast Cheybshev Transform
\cite{GottliebOrszag77,CHQZ88,Boyd01,ShenTang06} the
coefficients $\{ \hat{u}_{p,q,n} \}$ are readily recovered. These, in turn,
can be used to approximate the DNO, $\tG$.

%
%

\subsection{Pad\'e Approximation}
\label{Sec:Pade}

An important question is how the Taylor series, e.g.\ \eqref{Eqn:TFE:Approx:N},
in $\Eps$ is summed, for instance the approximation of 
$\hat{u}_{p,q}(\Eps)$ by
\bes
\hat{u}_{p,q}^N(\Eps) 
  := \sum_{n=0}^{N} \hat{u}_{p,q,n} \Eps^n.
\ees
As we have seen in \cite{BrunoReitich93b,NichollsReitich00b,Nicholls21}, 
Pad\'e approximation 
\cite{BakerGravesMorris96} has been used in conjunction with HOPS methods
with great success and we recommend its use here.  Pad\'e approximation
estimates the truncated Taylor series $\hat{u}_{p,q}^N(\Eps)$
by the rational function
\bes
[L/M](\Eps) := \frac{a^L(\Eps)}{b^M(\Eps)}
  = \frac{\sum_{\ell=0}^{L} a_{\ell} \Eps^{\ell}}
  {1 + \sum_{m=1}^{M} b_m \Eps^m},
\quad
L+M=N,
\ees
and
\bes
[L/M](\Eps) = \hat{u}_{p,q}^N(\Eps) + \BigOh{\Eps^{L+M+1}};
\ees
classical formulas for the coefficients $\{a_{\ell},b_m\}$
can be found in \cite{BakerGravesMorris96}.
This method has stunning properties of enhanced convergence,
and we refer the interested reader to \S~2.2 of Baker \& Graves--Morris
\cite{BakerGravesMorris96} and the insightful calculations
of \S~8.3 of Bender \& Orszag \cite{BenderOrszag78} for a
thorough discussion of the capabilities and limitations of
Pad\'e approximants.

%
%

\section{Numerical Results}
\label{Sec:NumRes}

We now present numerical results which not only demonstrate the validity 
of our implementations of the three algorithms presented above 
(OE, FE, and TFE) but also illuminate their characteristics and behavior.
With each of these we see that for small and smooth interfaces (e.g., 
$\tf$ analytic and $\Eps \ll 1$) all three algorithms deliver highly
accurate solutions in a stable and efficient manner. However, for large
and rough profiles (e.g., $\tf$ Lipschitz and $\Eps = \BigOh{1}$) the 
enhanced stability properties of the TFE method 
\cite{NichollsReitich00a,NichollsReitich03b} and superior accuracy
of Pad\'e summation become extremely important to realize accurate solutions.

%
%

\subsection{The Method of Manufactured Solutions}
\label{Sec:MMS}

To test the validity of our implementation, we used the
Method of Manufactured Solutions \cite{Burggraf66,Roache02,Roy05}.
To summarize this, consider the general system of
partial differential equations subject to generic boundary conditions
\begin{align*}
& \cP v = 0, && \text{in $\Omega$}, \\
& \cB v = 0, && \text{at $\partial \Omega$}.
\end{align*}
It is typically just as easy to implement a numerical algorithm to solve
the nonhomogeneous version of this set of equations
\begin{align*}
& \cP v = \cF, && \text{in $\Omega$}, \\
& \cB v = \cJ, && \text{at $\partial \Omega$}.
\end{align*}
To validate our code we began with the ``manufactured solution,''
$\tilde{v}$, and set
\bes
\cF_{\tilde v} := \cP \tilde{v},
\quad
\cJ_{\tilde v} := \cJ \tilde{v}.
\ees
Thus, given the pair $\{ \cF_{\tilde v}, \cJ_{\tilde v} \}$ we had an \textit{exact}
solution of the nonhomogeneous problem, namely $\tilde{v}$. While this
does not prove an implementation to be correct, if the function $\tilde{v}$
is chosen to imitate the behavior of anticipated solutions (e.g.,
satisfying the boundary conditions exactly) then this gives us
confidence in our algorithm.

%
%

\subsection{Two--Dimensional Simulations}
\label{Sec:2D}

For a  two--dimensional fluid of infinite depth ($h = \infty$)
we considered the laterally doubly $2 \pi$--periodic function
\bes
\tvarphi^\exct(\alpha,y) = A_q e^{\Abs{K^T q} y} e^{i q \cdot \alpha},
\quad
K \in \Real^{2 \times 1},
\quad
q = \begin{pmatrix} q_1 \\ q_2 \end{pmatrix} \in \Integer^2,
\ees
from which we readily computed
\bse
\label{Eqn:MMS:Input}
\begin{align}
\txi^\exct(\alpha) & = \tvarphi^\exct(\alpha,\tg(\alpha)), \\
\tnu^\exct(\alpha) & = \py \tvarphi^\exct(\alpha,\tg(\alpha))
  - \left(K^T \Grada \tg(\alpha)\right) \cdot 
  K^T \Grada{\tvarphi^\exct(\alpha,\tg(\alpha))},
\end{align}
\ese
and, from these,
\bes
\varphi^\exct(x,y) = \tvarphi^\exct(K x,y),
\quad
\xi^\exct(x) = \txi^\exct(K x),
\quad
\nu^\exct(x) = \tnu^\exct(K x).
\ees
We chose the following physical parameters
\begin{gather}
A_q = -3,
\quad
q = \begin{pmatrix} 1 \\ 1 \end{pmatrix},
\quad
K = \begin{pmatrix} 1 \\ \kappa \end{pmatrix},
\quad
\kappa = \sqrt{2},
\label{Eqn:Params:Phys}
\end{gather}
and the numerical parameters
\be
\label{Eqn:Params:Num}
N_{\alpha} = \begin{pmatrix} 64 \\ 64 \end{pmatrix},
\quad
N_y = 16,
\quad
N = 16,
\quad
a = \frac{1}{10}.
\ee
To illuminate the behavior of our scheme we considered
\bes
\tg(\alpha) = \Eps \tf(\alpha),
\quad
\tf(\alpha) = \cos(\alpha_1) \sin(\alpha_2),
\ees
so that $g(x) = \Eps f(x)$, and studied three choices
\bes
\Eps = 0.02, 0.5, 1.
\ees
For this we supplied the ``exact'' input data, $\txi^\exct$,
from \eqref{Eqn:MMS:Input} to our HOPS algorithm and compared
the output of this, $\tnu^{\text{approx}}$, with the 
``exact'' output, $\tnu^\exct$, by computing the relative error
\be
\label{Eqn:RelErr}
\text{Error}_{\text{rel}} :=
  \frac{\SupNorm{\tnu^\exct - \tnu^{\text{approx}}}}
  {\SupNorm{\tnu^\exct}}.
\ee

To evaluate our implementation and demonstrate the behavior of each
of the three schemes, we report our results in
Figures~\ref{Fig:MMS:Small}, \ref{Fig:MMS:Medium}, and \ref{Fig:MMS:Large}.
More specifically, Figure~\ref{Fig:MMS:Small} 
($\Eps = 0.02$) shows both the rapid and stable decay of the relative error
for this small perturbation as $N$ is increased for all three algorithms.
By contrast, Figure~\ref{Fig:MMS:Medium} ($\Eps = 0.5$)
shows the extremely beneficial effects of using the TFE algorithm
for this much larger deformation, particularly in concert with Pad\'e
summation. Finally, in Figure~\ref{Fig:MMS:Large} ($\Eps = 1$)
we see the \textit{necessity} of using both the TFE method and Pad\'e 
summation to realize any accuracy for such a huge perturbation.
%
%
%
%
%
%
%
%
%
%
%
%
\begin{figure}[hbt]
\centering
\includegraphics[width=0.5\textwidth]{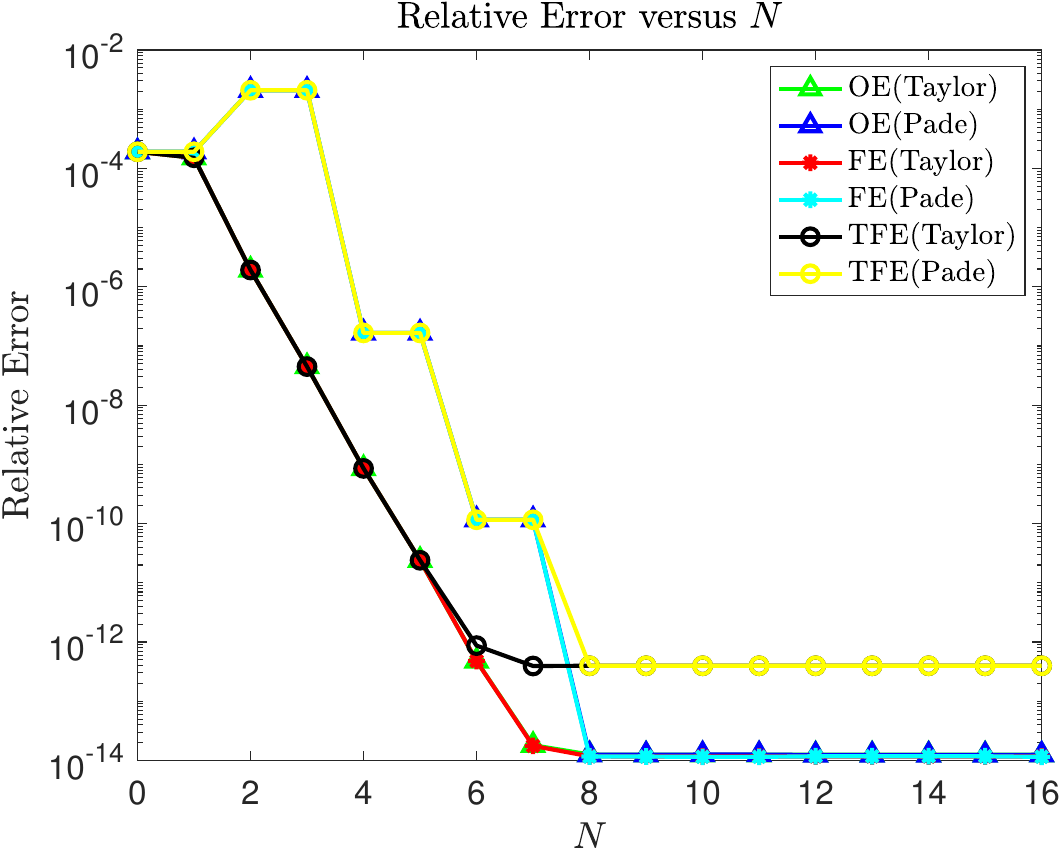}
\caption{Plot of relative error, \eqref{Eqn:RelErr}, for a small perturbation
  ($\Eps=0.02$) in the surface Neumann
  data for the three HOPS algorithms (OE, FE, TFE) using both Taylor and
  Pad\'e summation. Physical parameters were \eqref{Eqn:Params:Phys}
  and numerical discretization was \eqref{Eqn:Params:Num}.}
\label{Fig:MMS:Small}
\end{figure}
\begin{figure}[hbt]
\centering
\includegraphics[width=0.5\textwidth]{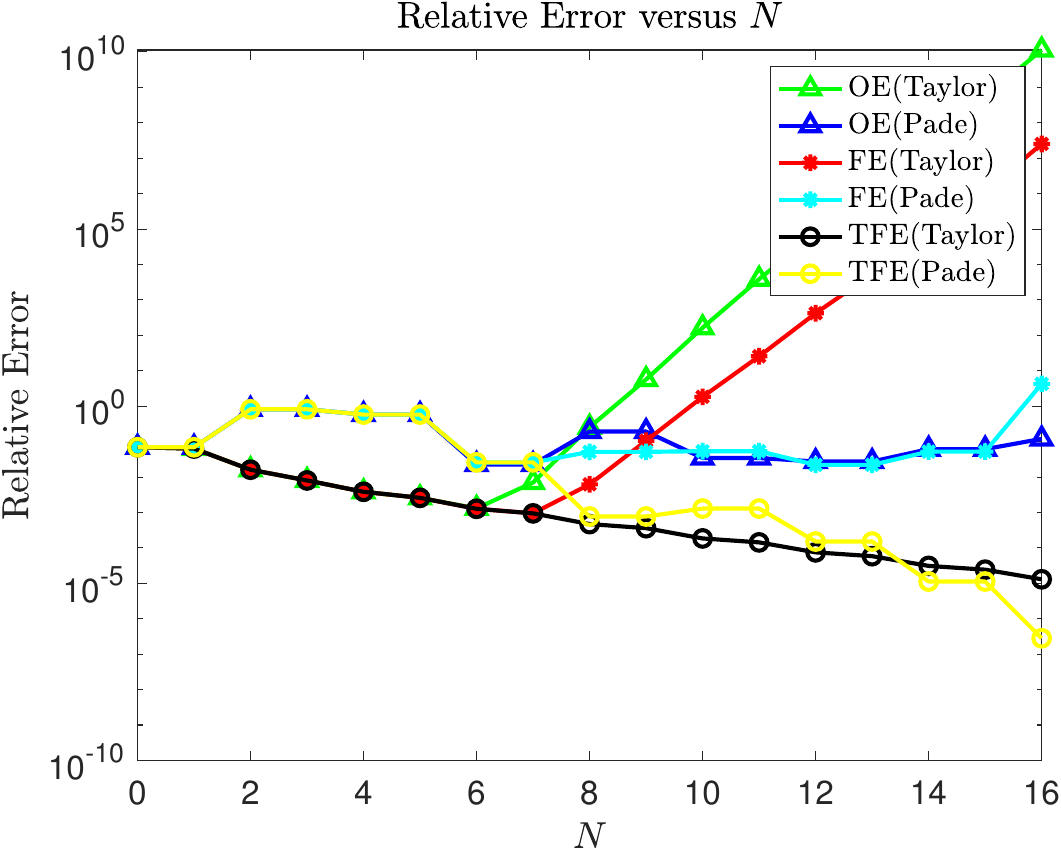}
\caption{Plot of relative error, \eqref{Eqn:RelErr}, for a medium perturbation
  ($\Eps=0.5$) in the surface Neumann
  data for the three HOPS algorithms (OE, FE, TFE) using both Taylor and
  Pad\'e summation. Physical parameters were \eqref{Eqn:Params:Phys}
  and numerical discretization was \eqref{Eqn:Params:Num}.}
\label{Fig:MMS:Medium}
\end{figure}
\begin{figure}[hbt]
\centering
\includegraphics[width=0.5\textwidth]{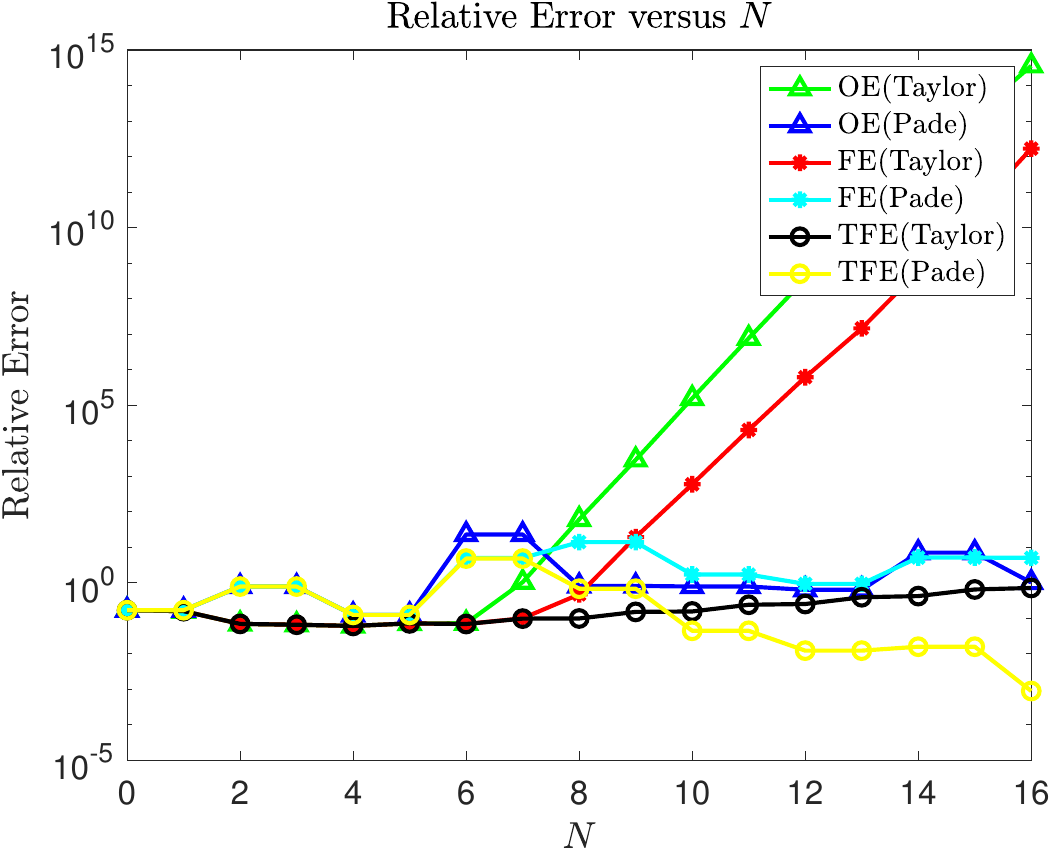}
\caption{Plot of relative error, \eqref{Eqn:RelErr}, for a large perturbation
  ($\Eps=1$) in the surface Neumann
  data for the three HOPS algorithms (OE, FE, TFE) using both Taylor and
  Pad\'e summation. Physical parameters were \eqref{Eqn:Params:Phys} 
  and numerical discretization was \eqref{Eqn:Params:Num}.}
\label{Fig:MMS:Large}
\end{figure}

%
%

\subsection{Three--Dimensional Simulations}
\label{Sec:3D}

We begin a description of our three--dimensional numerical simulations
with the infinite depth case ($h = \infty$) where we chose
the following triply $2 \pi$-periodic function
\bes
\tvarphi^\exct(\alpha,y) = \left\{ A_q e^{i q \cdot \alpha} 
  + \bar{A}_q e^{-i q \cdot \alpha} \right\} e^{\Abs{K^T q} y},
\quad
K \in \Real^{3 \times 2},
\quad
q = \begin{pmatrix} q_1 \\ q_2 \\ q_3 \end{pmatrix} \in \Integer^3.
\ees
As before, with this we can generate all of $\{ \txi^\exct, \tnu^\exct \}$ and 
$\{ \varphi^\exct, \xi^\exct, \nu^\exct \}$. We chose the physical parameters
\be
\label{Eqn:Params:Phys:3d}
A_q = -1, 
\quad 
K = \begin{pmatrix}
1 & 0\\
0 & -1\\
1/\sqrt{2} & 1/\sqrt{3}
\end{pmatrix},
\quad
q = \begin{pmatrix} 1 \\ 1 \\ 2 \end{pmatrix},
\ee
and the numerical parameters
\be
\label{Eqn:Params:Num:3d}
N_{\alpha} = \begin{pmatrix} 128 \\ 128 \\128 \end{pmatrix},
\quad
N_y = 128, 
\quad
N = 32,
\quad
a = 0.5.
\ee
Finally, we considered the profile
\bes
\tg(\alpha) = \Eps \tf(\alpha), \quad
\tf(\alpha) = \cos(\alpha_1) + \cos(\alpha_2) + \sin(\alpha_3).
\ees
We summarize our results in Figure~\ref{Fig:3D:h_infty} 
which focuses on the TFE algorithm utilizing Taylor summation alone.
Here we see that for the small perturbation, $\Eps=0.02$, we see
stable and rapid convergence to nearly machine zero after merely 4--5
perturbation orders. In the moderate deformation case, $\Eps=0.1$, we again
see the speedy convergence of our algorithm which achieves machine
precision by 15--16 perturbation orders. From extensive experimentation
we learned that $\Eps=0.02$ is \textit{outside} the disk of convergence
of our perturbation series which explains the slow, but noticeable divergence
of our results as $N$ is increased. (In fact, our experiments show that
$\Eps = 0.16$ is \textit{inside} the disk while $\Eps = 0.17$ is \textit{outside}.)
\begin{figure}[hbt]
\centering
\includegraphics[width=0.5\textwidth]{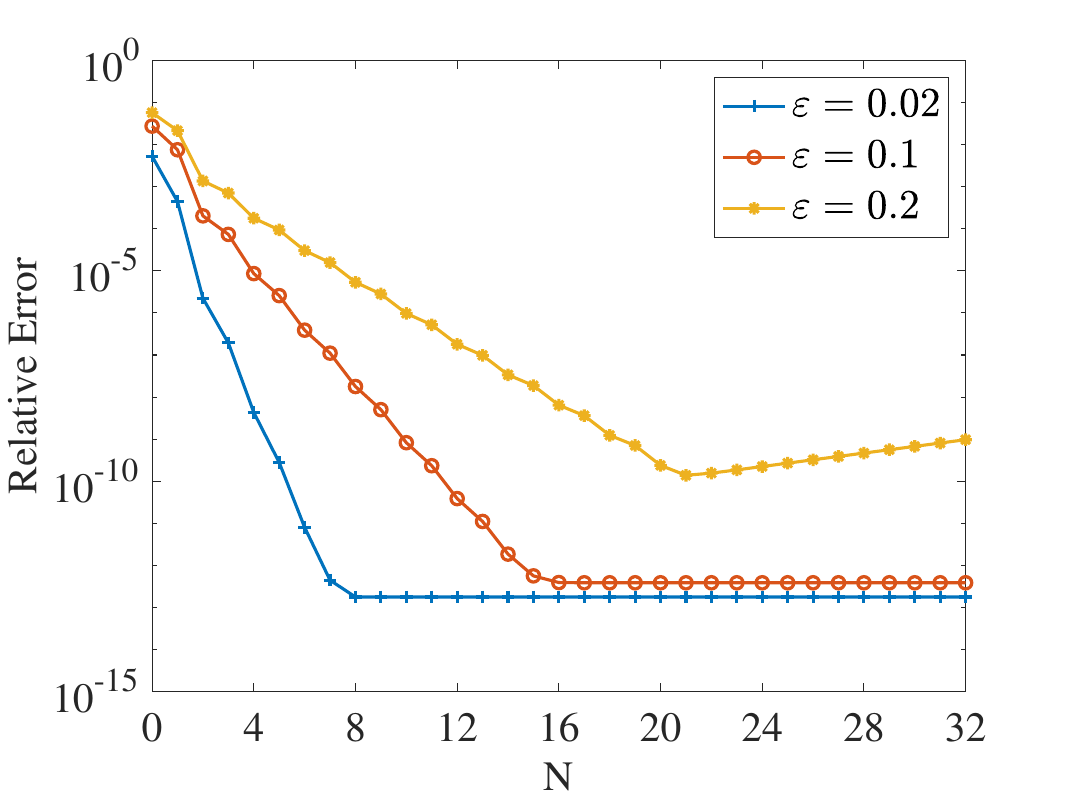}
\caption{Plot of relative error, \eqref{Eqn:RelErr}, in the surface Neumann
  data for the TFE algorithms using Taylor summation for perturbation
  sizes $\Eps = 0.01, 0.1, 0.2$. Physical parameters were 
  \eqref{Eqn:Params:Phys:3d} and numerical
  discretization was \eqref{Eqn:Params:Num:3d}.}
\label{Fig:3D:h_infty}
\end{figure}
To further illustrate the capabilities of our implementation we conducted
this simulation in the case $\Eps = 0.16$, right at the boundary of
convergence of the Taylor series \eqref{Eqn:tG:Exp}. In Figure~\ref{Fig:Field:3D}
on the left we show a plot of our simulation of the Neumann data, $\tnu(\alpha)$.
On the right of this we depict a slice of this Neumann data with a plane
that has normal vector $(1/\sqrt{2},-1/\sqrt{3},-1)^T$.
\begin{figure}
\centering
  \begin{subfigure}[b]{0.49\textwidth}
  \centering \includegraphics[width=\textwidth]{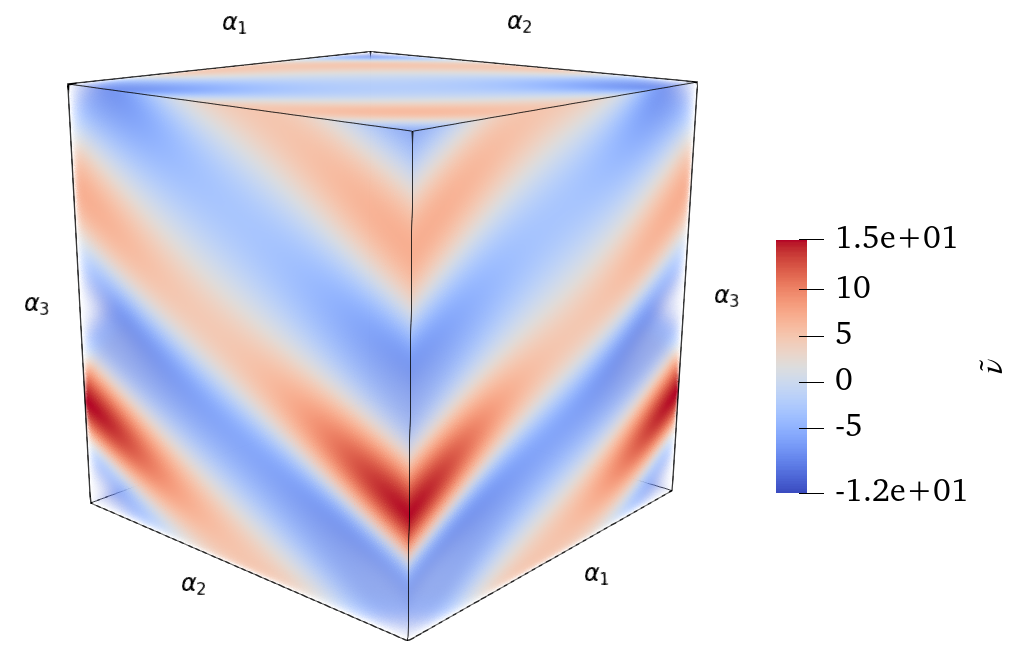}  
  \end{subfigure}
  \hfill
  \begin{subfigure}[b]{0.49\textwidth}
  \centering \includegraphics[width=\textwidth]{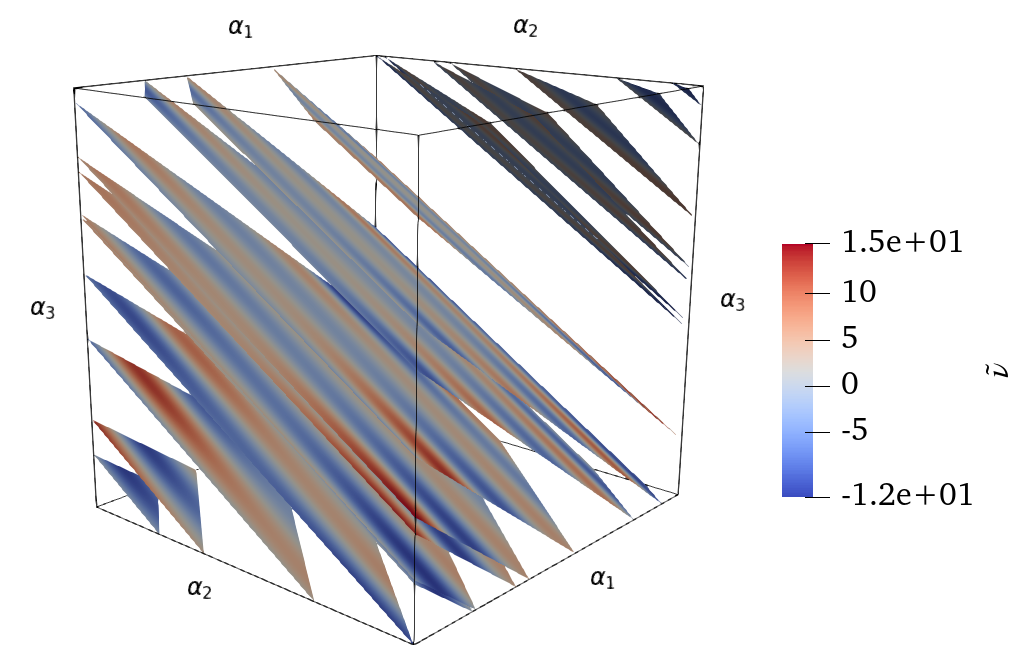}
  \end{subfigure}
  \caption{Plot of the computed Neumann data, $\tu$, for $\Eps=0.16$.
    (a) The full field; (b) A slice of the full field with a plane that has
    normal vector $(1/\sqrt{2}, -1/\sqrt{3}, -1)^T$.}
\label{Fig:Field:3D}
\end{figure}
This simulation produced Taylor corrections $\tnu_n$ and we depict in
Figure~\ref{Fig:Field:3D:n} contour plots of $\tnu_0$ (panel (a)),
$\tnu_2$ (panel (b)), $\tnu_4$ (panel (c)), and $\tnu_6$ (panel (d)).
\begin{figure}
\centering
\begin{subfigure}[b]{0.47\textwidth}
  \centering \includegraphics[width=\textwidth]{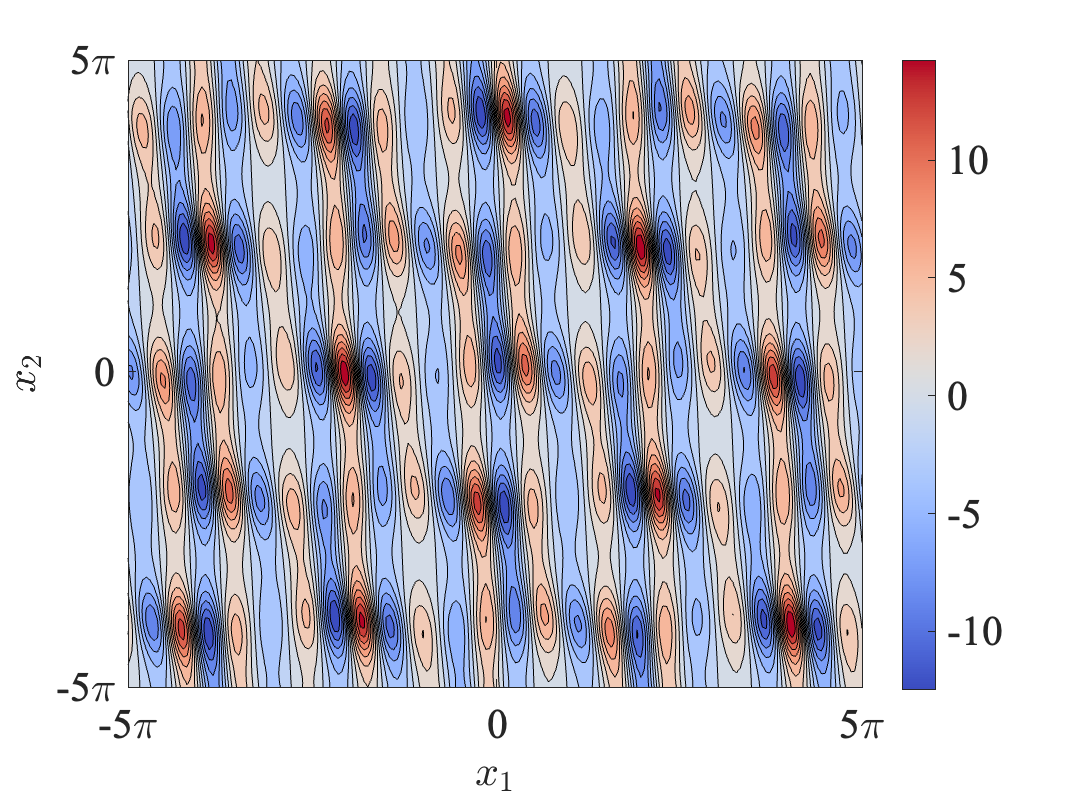} 
  \caption{$n=0$}
\end{subfigure}
\hfill
\begin{subfigure}[b]{0.47\textwidth}
  \centering \includegraphics[width=\textwidth]{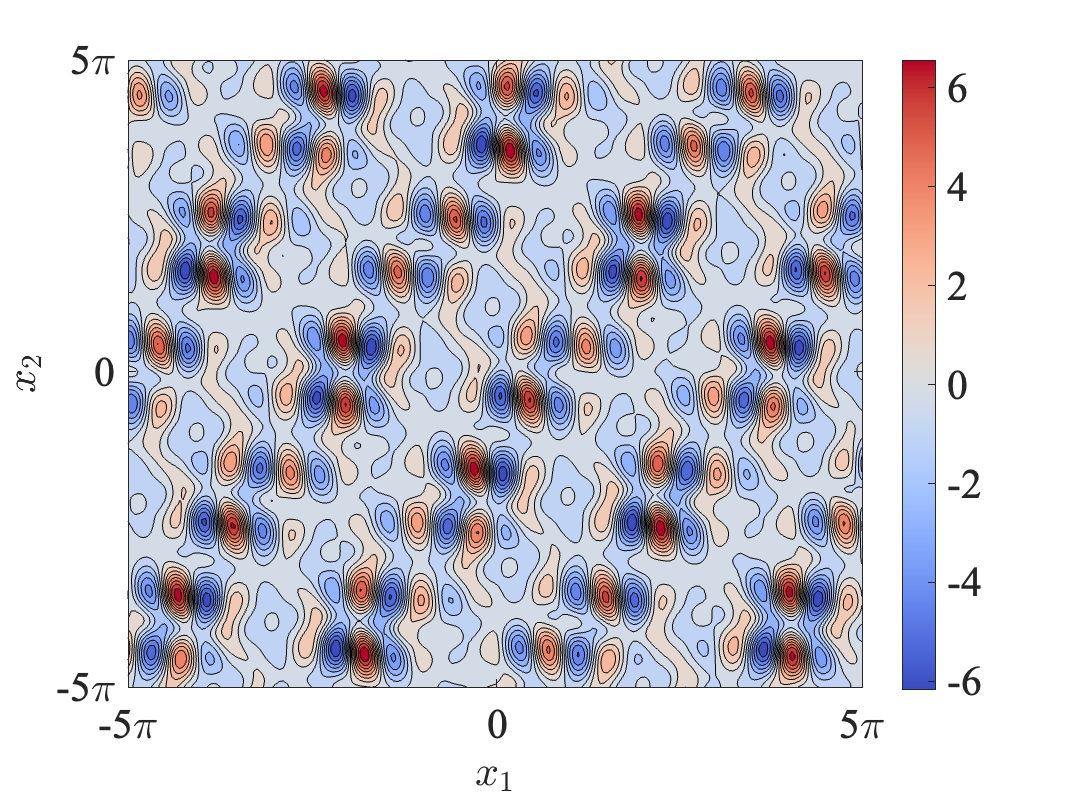}
  \caption{$n=2$}
\end{subfigure}

\begin{subfigure}[b]{0.47\textwidth}
  \centering \includegraphics[width=\textwidth]{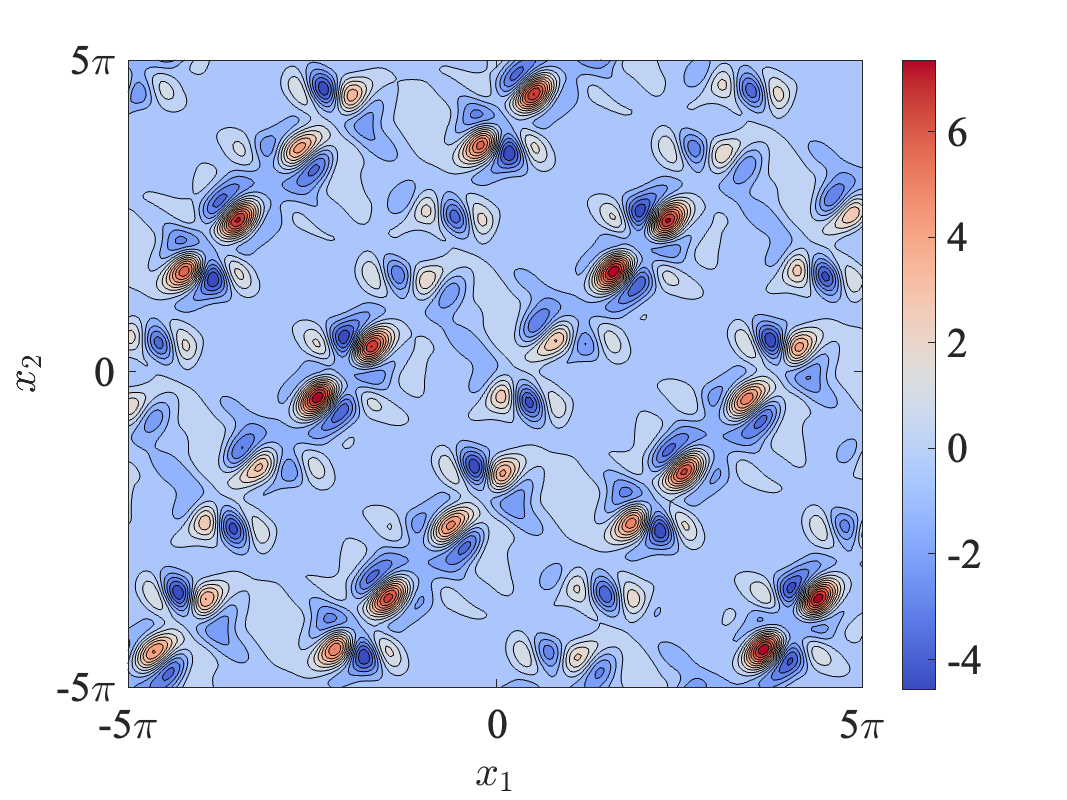} 
  \caption{$n=4$}
\end{subfigure}
\hfill
\begin{subfigure}[b]{0.47\textwidth}
  \centering \includegraphics[width=\textwidth]{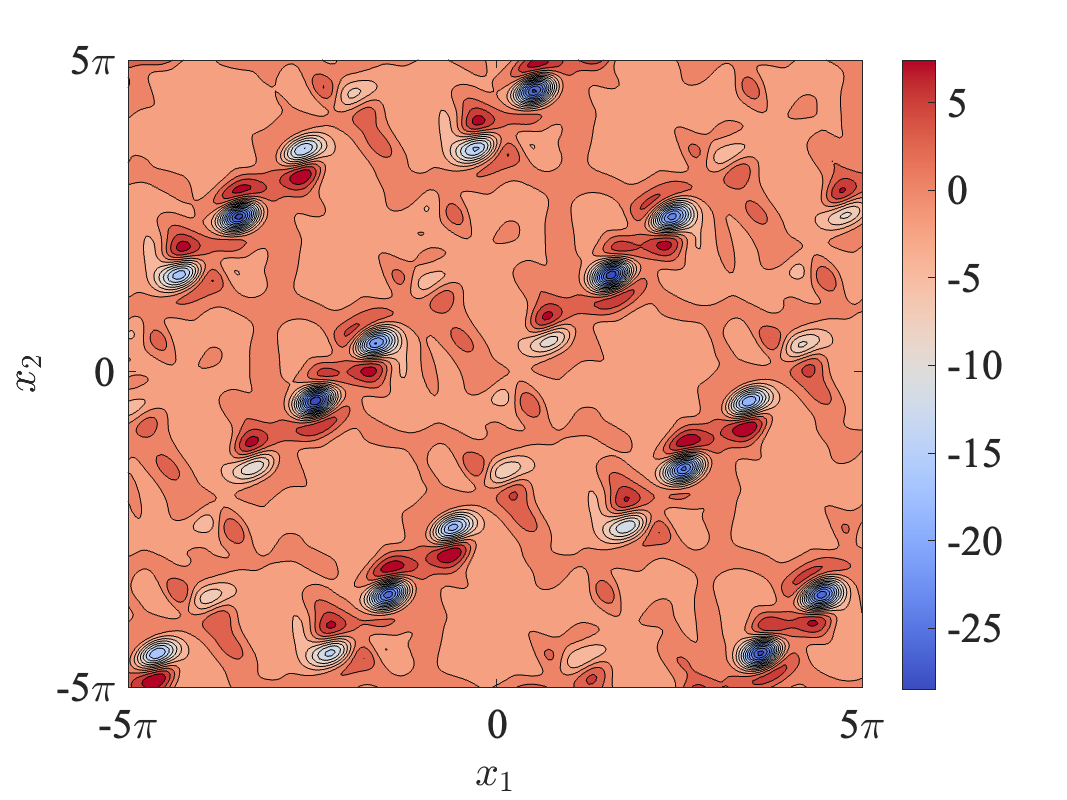}
  \caption{$n=6$}
\end{subfigure}   
\caption{Contour plots of the Taylor corrections $\tnu_n$ for
  $n=0$ (panel (a)), $n=2$ (panel (b)), $n=4$ (panel (c)),
  and $n=6$ (panel (d)).}
\label{Fig:Field:3D:n}
\end{figure}

We conclude with a description of our three--dimensional numerical
simulations in the finite depth case ($h = 1/4$) where we chose
the following triply $2 \pi$-periodic function
\bes
\tvarphi^\exct(\alpha,y) = \left\{ A_q e^{i q \cdot \alpha} 
  + \bar{A}_q e^{-i q \cdot \alpha} \right\}
  \frac{\cosh(\Abs{K^T q} (h+y))}{\cosh(\Abs{K^T q} h)},
\quad
K \in \Real^{3 \times 2},
\quad
q = \begin{pmatrix} q_1 \\ q_2 \\ q_3 \end{pmatrix} \in \Integer^3.
\ees
Once again, with this we can generate all of $\{ \txi^\exct, \tnu^\exct \}$ and 
$\{ \varphi^\exct, \xi^\exct, \nu^\exct \}$. We chose the physical parameters
\be
\label{Eqn:Params:Phys:3d:Finite}
A_q = -1, 
\quad 
K = \begin{pmatrix}
1 & 0\\
0 & -1\\
1/\sqrt{2} & 1/\sqrt{3}
\end{pmatrix},
\quad
q = \begin{pmatrix} 1 \\ 1 \\ 2 \end{pmatrix},
\ee
and the numerical parameters
\be
\label{Eqn:Params:Num:3d:Finite}
N_{\alpha} = \begin{pmatrix} 128 \\ 128 \\128 \end{pmatrix},
\quad
N_y = 128, 
\quad
N = 32,
\quad
a = 0.1.
\ee
Additionally, we considered the profile
\bes
\tg(\alpha) = \Eps \tf(\alpha), \quad
\tf(\alpha) = \cos(\alpha_1) + \cos(\alpha_2) + \sin(\alpha_3).
\ees
In Figure~\ref{Fig:3D:h_1} we display results from our simulations
using the TFE algorithm with Taylor summation.
As before, we see that for the small perturbation, $\Eps=0.02$, we have
stable and rapid convergence to nearly machine zero after merely 11--12
perturbation orders.
\begin{figure}[hbt]
\centering
\includegraphics[width=0.5\textwidth]{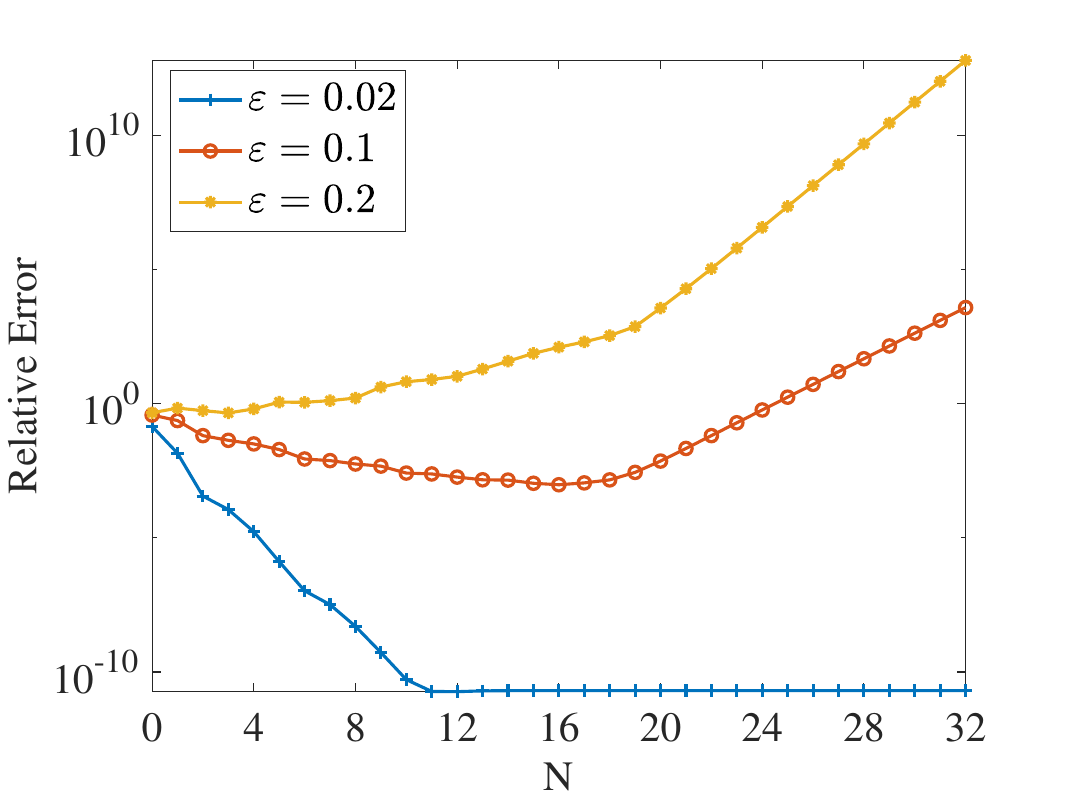}
\caption{Plot of relative error, \eqref{Eqn:RelErr}, in the surface Neumann
  data for the TFE algorithms using Taylor summation for perturbation
  sizes $\Eps = 0.01, 0.1, 0.2$. Physical parameters were 
  \eqref{Eqn:Params:Phys:3d:Finite} and numerical
  discretization was \eqref{Eqn:Params:Num:3d:Finite}.}
\label{Fig:3D:h_1}
\end{figure}
%

%
%

\section*{Acknowledgments}

D.P.N.\ gratefully acknowledges support from the National Science
Foundation through grant No.~DMS--1813033 and No.~DMS--2111283. J.W.\ gratefully acknowledges support from the National
  Science Foundation under award number DMS-1716560 and from the
  Department of Energy, Office of Science, Applied Scientific
  Computing Research, under award number DE-AC02-05CH11231. X.Z.\ gratefully acknowledges the high-performance computing resources provided by Digital Research Alliance of Canada.

%
%

\appendix

%
%

\section{Proof of the Elliptic Estimate}
\label{Sec:EllEst:Proof}

In this appendix we take up the proof of Theorem~\ref{Thm:EllEst}
which we restate here for completeness. Recall that we consider
the generic boundary value problem, \eqref{Eqn:BVP:QP},
\begin{align*}
& \Diva{ K K^T \Grada{\tu(\alpha,y)} } + \py^2 \tu(\alpha,y) 
  = \tF(\alpha,y), && -a < y < 0, \\
& \tu(\alpha,0) = \txi(\alpha), && \\
& \py \tu(\alpha,-a) - \tT[\tu(\alpha,-a)] = \tJ(\alpha), &&
\end{align*}
where
\bes
\tF(\alpha,y) := \Diva{ K \tFa(\alpha,y) } + \py \tF^y(\alpha,y) 
  + \tF^0(\alpha,y).
\ees
Recalling the maximum,
\begin{align*}
\cM_s[\tu] := \max \left\{ \right.
  &\XNorm{\tu}{s},
  \XNorm{\py \tu}{s},
  \XNorm{K^T \Grada{\tu}}{s},
  \\ 
  &\XNorm{\AKD^{1/2} \tu}{s},
  \XNorm{\AKD^{1/2} \py \tu}{s},
  \XNorm{\AKD^{1/2} K^T \Grada{\tu}}{s},
 \\ 
  &\left.\SobNorm{\py \tu(\alpha,0)}{s},
  \SobNorm{K^T \Grada{\tu}(\alpha,0)}{s},
  \SobNorm{\tT[\tu(\alpha,-a)]}{s} \right\}.
\end{align*}
the elliptic estimate we require (Theorem~\ref{Thm:EllEst})
can be stated as follows.
\begin{Thm}
Given an integer $s \geq 0$, provided that
\begin{gather*}
\txi \in H^{s+1}(P(\Gamma)),
\quad
\tJ \in H^s(P(\Gamma)),
\\
\tF^j, \AKD^{1/2} \tF^j \in X^s(\Omega),
\quad
\tF^j(\alpha,0) \in H^s(P(\Gamma)),
\quad
j \in \{ \alpha, y, 0 \},
\end{gather*}
and
\bes
\tF^y(\alpha,-a) = 0,
\ees
there exists a unique solution $\tu \in X^s(\Omega)$ of \eqref{Eqn:BVP:QP}
such that, for some $C_e > 0$,
\begin{multline*}
\cM_s[\tu] \leq C_e \left\{ \SobNorm{\txi}{s+1}
  + \SobNorm{\tJ}{s} 
\right. \\ \left.
  + \sum_{j \in \{ \alpha, y, 0 \}}
  \left( \XNorm{\tF^j}{s}
  + \XNorm{\AKD^{1/2} \tF^j}{s} 
  + \SobNorm{ \tF^j(\alpha,0) }{s} \right)
  \right\}.
\end{multline*}
\end{Thm}
\begin{proof}
We will focus on the case $h = \infty$ and note that the proof for the case
$h < \infty$ follows similarly. We expand the data and solution in Fourier series as,
\bes
\{ \tu, \tF^j \}(\alpha,y) = \sump \left\{ \hat{u}_p(y), \hat{F}^j_p(y) \right\}
  e^{i p \cdot \alpha},
\quad
\{ \txi, \tJ \}(\alpha) = \sump \left\{ \hat{\xi}_p, \hat{J}_p \right\}
  e^{i p \cdot \alpha},
\ees
giving,
\begin{align*}
& \py^2 \hat{u}_p(y) - \AKp^2 \hat{u}_p(y)
  = (i K^T p) \cdot \hFa_p(y) + \py \hat{F}_p^y(y)
  + \hat{F}_p^0(y),
  && -a < y < 0, \\
& \hat{u}_p(0) = \hat{\xi}_p, \\
& \py \hat{u}_p(-a) - \AKp \hat{u}_p(-a) = \hat{J}_p.
\end{align*}
To complete our proof we work with the exact solutions of these
which we now derive and estimate explicitly.
\begin{description}
%
%
\item[(a.) Inhomogeneous BCs:]
We begin with the case $\hat{F}^j \equiv 0$ and write
\bes
\hat{u}_p(y) = \hat{\xi}_p E_p(y) + \hat{J}_p S_p(y),
\ees
where
\bes
E_p(y) := \exp(\AKp y),
\quad
S_p(y) := 
\begin{cases}
y, & p = 0, \\
\frac{\sinh(\AKp y)}{\AKp \exp(\AKp a)}, & p\neq 0.
\end{cases}
\ees
From this we compute
\bes
\py \hat{u}_p(y) = \AKp \hat{\xi}_p E_p(y) + \hat{J}_p C_p(y),
\ees
where
\bes
C_p(y) := 
\begin{cases}
1, & p = 0, \\
\frac{\cosh(\AKp y)}{\exp(\AKp a)}, & p \neq 0.
\end{cases}
\ees
Computing the $L^2$ norms of $E_p(y)$, $S_p(y)$ and $C_p(y)$, we obtain
\begin{align*}
\Norm{E_p}{L^2}^2 & = \frac{1-\exp(-2 \AKp a)}{2 \AKp}, \\
\Norm{S_p}{L^2}^2 & = \begin{cases}
  a^3/3, & p = 0, \\
  \frac{-2 a \Abs{K^T p} + \sinh(2 a \Abs{K^T p})}
  {4 \exp(2 a \Abs{K^T p}) \Abs{K^T p}^3}, & p \neq 0, 
  \end{cases} \\
\Norm{C_p}{L^2}^2 & = \begin{cases}
  a, & p = 0, \\
  \frac{2 a \Abs{K^T p} + \sinh(2 a \Abs{K^T p})}
  {4 \exp(2 a \Abs{K^T p}) \Abs{K^T p}}, & p \neq 0,
  \end{cases}
\end{align*}
which have the following limits when $p \neq 0$: 
As $\AKp \rightarrow 0$ we have
\bes
\Norm{E_p}{L^2}^2 \rightarrow a,
\quad
\Norm{S_p}{L^2}^2 \rightarrow \frac{a^3}{3},
\quad
\Norm{C_p}{L^2}^2 \rightarrow a,
\ees
and, as $\AKp \rightarrow \infty$ we find
\bes
\Norm{E_p}{L^2}^2 \rightarrow \frac{1}{2} \AKp^{-1},
\quad
\Norm{S_p}{L^2}^2 \rightarrow \frac{1}{4} \AKp^{-3},
\quad
\Norm{C_p}{L^2}^2 \rightarrow \frac{1}{4} \AKp^{-1}.
\ees
From these, through interpolation it is not difficult to establish
the following result.
\begin{Lemma}
\label{lem:xi:J}
\begin{align*}
& \AKp^m \Norm{E_p}{L^2}^2 < C,
  && 0 \leq m \leq 1, \quad \forall\ p \in \Gamma', \\
& \AKp^m \Norm{S_p}{L^2}^2 < C,
  && 0 \leq m \leq 3, \quad \forall\ p \in \Gamma', \\
& \AKp^m \Norm{C_p}{L^2}^2 < C,
  && 0 \leq m \leq 1, \quad \forall\ p \in \Gamma',
\end{align*}
\end{Lemma}
In addition, we note that
\begin{gather*}
\hat{u}_p(0) = \hat{\xi}_p,
\quad
\py \hat{u}_p(0) = \AKp \hat{\xi}_p
  + \hat{J}_p \exp(-\AKp a),
\\
\hat{u}_p(-a) = \begin{cases}
  \hat{\xi}_0 - a \hat{J}_0, & p = 0, \\
  \hat{\xi}_p \exp(-\AKp a)
  - \hat{J}_p \frac{(1-\exp(-2 \AKp a))}{2 \AKp}, & p \neq 0.
  \end{cases}
\end{gather*}
We now estimate the eight terms in $\cM_s(\tu)$.
  \begin{description}
  %
  %
  \item[(i.) Field:]
  \begin{align*}
  \XNorm{\tu}{s}^2 & = \sump \Angle{p}^{2s} 
      \Norm{\hat{\xi}_p E_p(y) + \hat{J}_p S_p(y)}{L^2}^2 \\
    & \leq \sump 2 \Angle{p}^{2s} \left( \Abs{\hat{\xi}_p}^2 
    \Norm{E_p(y)}{L^2}^2 + \Abs{\hat{J}_p}^2 \Norm{S_p(y)}{L^2}^2 \right) \\
    & \leq 2 C \left(\SobNorm{\txi}{s}^2 + \SobNorm{\tJ}{s}^2\right),
  \end{align*}
  where the last inequality follows from Lemma~\ref{lem:xi:J}.
 In a similar manner we can show that
  \bes
  \XNorm{\AKD^{1/2} \tu}{s} \leq 2 C \left(\SobNorm{\txi}{s}^2 
    + \SobNorm{\tJ}{s}^2\right).
  \ees
  %
  %
  \item[(ii.) Vertical derivative:] 
  \begin{align*}
  \XNorm{\py \tu}{s}^2 & = \sump \Angle{p}^{2s} 
      \Norm{\AKp \hat{\xi}_p E_p(y) + \hat{J}_p C_p(y)}{L^2}^2 \\
    & \leq \sump 2 \Angle{p}^{2s} \left( \AKp\Abs{\hat{\xi}_p}^2 
      \AKp \Norm{E_p(y)}{L^2}^2
      + \Abs{\hat{J}_p}^2 \Norm{C_p(y)}{L^2}^2 \right) \\
    & \leq 2 C C' \left( \SobNorm{\txi}{s+1/2}^2 
      + \SobNorm{\tJ}{s}^2 \right),
  \end{align*}
  where the last inequality follows from Lemma~\ref{lem:xi:J}
  and $\AKp \leq C' \Abs{p}$ for some $C'\geq 1$ and for any $p \in \Gamma'$.
  Similarly, one can show that 
  \begin{align*}
   & \XNorm{\AKD^{1/2} \py \tu}{s}^2 \leq 2 C C'^2 
     \left( \SobNorm{\txi}{s+1}^2 + \SobNorm{\tJ}{s}^2 \right), \\
   & \SobNorm{\py \tu (0)}{s}^2 \leq 2 C'^2 
     \left( \SobNorm{\txi}{s+1}^2 + \SobNorm{\tJ}{s}^2 \right).
  \end{align*}
  %
  %
  \item[(iii.) Horizontal derivative:] 
  \begin{align*}
  \XNorm{K^T \Grada{\tu}}{s}^2
    & = \sump \Angle{p}^{2s} \AKp^2 \Norm{\hat{\xi}_p E_p(y) + 
     \hat{J}_p S_p(y)}{L^2}^2 \\
    & \leq \sump 2 \Angle{p}^{2s} 
      \left( \AKp\Abs{\hat{\xi}_p}^2 \AKp\Norm{E_p(y)}{L^2}^2
    \right. \\ & \quad \left.
      + \Abs{\hat{J}_p}^2 \AKp^2\Norm{S_p(y)}{L^2}^2 \right) \\
    & \leq 2 C C'\left( \SobNorm{\txi}{s+1/2}^2 
      + \SobNorm{\tJ}{s}^2\right).
  \end{align*}
  Similarly, one can prove that
  \begin{align*}
  & \XNorm{\AKD^{1/2} K^T \Grada{\tu}}{s}^2 
    \leq 2C C'^2\left(\SobNorm{\txi}{s+1}^2 + \SobNorm{\tJ}{s}^2\right), \\
  & \SobNorm{K^T\Grada{\tu}(0)}{s}^2 \leq C'^2 \SobNorm{\txi}{s+1}^2.
  \end{align*}
  %
  %
  \item[(iv.) DNO at the lower boundary:]
  Moving to $y=-a$ we have
  \begin{align*}
  \SobNorm{\tT[ \tu(-a) ]}{s}^2
    & = \sump \Angle{p}^{2s} \Abs{ \hat{\xi}_p \AKp \exp(-\AKp a)
    \right. \\ & \quad \left.
      - \frac{1}{2} \hat{J}_p (1-\exp(-2 \AKp a)) }^2 \\
    & \leq \sump 2 \Angle{p}^{2s} \left(\AKp^2 \Abs{\hat{\xi}_p}^2
      + \Abs{\hat{J}_p}^2 \right) \\
    & \leq 2 C'^2 \left( \SobNorm{\txi}{s+1}^2 + \SobNorm{\tJ}{s}^2
      \right).
  \end{align*}
  \end{description}
%
%
\item[(b.) Inhomogeneous PDE  ($\tF^0 \not \equiv 0$):]
We move to the case 
$\hat{\xi}_p \equiv \hat{J}_p \equiv 0$ and
$\hat{F}^j \equiv 0$, $j \in \{\alpha, y\}$.
For $\tF = \tF^0$, we have
\bes
\hat{u}_p(y) = 
  \exp(\AKp a) \left\{ \int_y^0 E_p(y) S_p(t) \hat{F}_p(t) \dt
    + \int_{-a}^y E_p(t) S_p(y) \hat{F}_p(t) \dt \right\}.
\ees
In Appendix~\ref{App:InhomogPDE} we prove the following
estimates of this function and its derivatives.
\begin{Lemma}
\label{lem:F}
\begin{align*}
& \AKp^m \Norm{\hat u_p}{L^2} \leq C \Norm{\hat{F}_p}{L^2},
  && 0 \leq m \leq 2, & \forall\ p \in \Gamma', \\
& \AKp^m \Norm{\py\hat u_p}{L^2} \leq C \Norm{\hat{F}_p}{L^2},
  && 0 \leq m \leq 1, & \forall\ p \in \Gamma', \\
& \AKp^m \Abs{\py\hat u_p(0)} \leq C \Norm{\hat{F}_p}{L^2},
  && 0 \leq m \leq 1/2, & \forall\ p \in \Gamma', \\
& \AKp^m \Abs{\hat u_p(-a)} \leq C \Norm{\hat{F}_p}{L^2},
  && 0 \leq m \leq 3/2, & \forall\ p \in \Gamma'.
\end{align*}
\end{Lemma}
Using Lemma~\ref{lem:F} we can prove the following estimates in a 
rather straightforward way.
  \begin{description}
  %
  %
  \item[(i.) Field:]
  \bes
  \max \left\{ \XNorm{\tu}{s}, \XNorm{\AKD^{1/2} \tu}{s} \right\}
    \leq C \XNorm{\tF^0}{s}.
  \ees
  %
  %
  \item[(ii.) Vertical derivative:] 
  \bes
  \max \left\{ \XNorm{\py\tu}{s}, \XNorm{\AKD^{1/2} \py \tu}{s},
  \SobNorm{\py \tu (0)}{s}^2 \right\} \leq C \XNorm{\tF^0}{s}.
  \ees
  %
  %
  \item[(iii.) Horizontal derivative:] 
  \bes
  \max \left\{ \XNorm{K^T \Grada{\tu}}{s},
  \XNorm{\AKD^{1/2}K^T \Grada{\tu}}{s}, 
  \SobNorm{K^T \Grada{\tu}(0)}{s} \right\}
    \leq C \XNorm{\tF^0}{s}.
  \ees
  %
  %
  \item[(iv.) DNO at the lower boundary:]
  \bes
  \SobNorm{\tT[ \tu(-a) ]}{s} 
  = \SobNorm{\AKD\tu(-a)}{s} \leq C \XNorm{\tF^0}{s}.
  \ees
  \end{description}
%
%
\item[(c.) Inhomogeneous PDE ($\tFa \not \equiv 0$):]
We move to the case $\hat{\xi}_p \equiv \hat{J}_p \equiv 0$ and
$\hat{F}^j \equiv 0$, $j \in \{0, y\}$. Setting $\tF = \tFa$, we have
\begin{align*}
\hat{u}_p(y) & = \exp(\AKp a) \left\{
  \int_y^0 E_p(y) S_p(t) (i K^T p) \hat{F}_p(t) \dt 
  \right. \\ & \quad \left.
  + \int_{-a}^y E_p(t) S_p(y) (i K^T p) \hat{F}_p(t) \dt \right\}.
\end{align*}
By simply replacing $\hat F_p$ by $(i K^T p) \hat{F}_p$ in 
Lemma~\ref{lem:F} we obtain the following estimates.
\begin{Lemma}
\begin{align*}
& \AKp^m \Norm{\hat u_p}{L^2} \leq C \Norm{\hat{F}_p}{L^2},
  && 0 \leq m \leq 1, & \forall\ p \in \Gamma', \\
& \Norm{\py\hat u_p}{L^2} \leq C \Norm{\hat{F}_p}{L^2},
  && & \forall\ p \in \Gamma', \\
& \Abs{\py\hat u_p(0)} \leq C \AKp^{1/2} \Norm{\hat{F}_p}{L^2},
  && & \forall\ p \in \Gamma', \\
& \AKp^m \Abs{\hat u_p(-a)} \leq C \Norm{\hat{F}_p}{L^2},
  && 0 \leq m \leq 1/2, & \forall\ p \in \Gamma'.
\end{align*}
\end{Lemma}
With this it is not difficult to show the following.
  \begin{description}
  %
  %
  \item[(i.) Field:]
  \bes
  \max \left\{ \XNorm{\tu}{s}, \XNorm{\AKD^{1/2} \tu}{s} \right\}
    \leq C \XNorm{\tFa}{s}.
  \ees
  %
  %
  \item[(ii.) Vertical derivative:] 
  \begin{align*}
  & \XNorm{\py\tu}{s} \leq C \XNorm{\tFa}{s}, \\
  & \max \left\{ \XNorm{\AKD^{1/2} \py \tu}{s},
  \SobNorm{\py \tu (0)}{s}^2 \right\} \leq C \XNorm{\AKD^{1/2}\tFa}{s}.
  \end{align*}
  %
  %
  \item[(iii.) Horizontal derivative:] 
  \begin{align*}
  & \XNorm{K^T \Grada{\tu}}{s} \leq C \XNorm{\tFa}{s}, \\
  & \left\{ \XNorm{\AKD^{1/2}K^T \Grada{\tu}}{s},
    \XNorm{K^T \Grada{\tu}(0)}{s} \right\}
    \leq C \XNorm{\AKD^{1/2}\tFa}{s}.
  \end{align*}
  %
  %
  \item[(iv.) DNO at the lower boundary:]
  \bes
  \SobNorm{\tT[ \tu(-a) ]}{s} 
  = \SobNorm{\AKD\tu(-a)}{s} \leq C \XNorm{\AKD^{1/2}\tFa}{s}.
  \ees
  \end{description}
%
%
\item[(d.) Inhomogeneous PDE ($\tF^y \not \equiv 0$):]
We move to the case the case $\hat{\xi}_p \equiv \hat{J}_p \equiv 0$ and
$\hat{F}^j \equiv 0$, $j \in \{0, \alpha\}$. For $\tF = \tF^y$, we have
\bes
\hat{u}_p(y) = \exp(\AKp a) \left\{
  \int_y^0 E_p(y) S_p(t) \hat{F}_p'(t) \dt
  + \int_{-a}^y E_p(t) S_p(y) \hat{F}'_p(t) \dt \right\}.
\ees
Writing this as 
\bes
\hat{u}_p(y) = I_1[\py \hat{F}_p](y) - I_2[\py \hat{F}_p](y) 
  + I_3[\py \hat{F}_p](y) - I_4[\py \hat{F}_p](y).
\ees
Using Integration--by--Parts, together with $\hat{F}^y_p(-a) = 0$,
we discover
\bes
\hat{u}_p(y) = \AKp \left\{-I_1[\hat{F}_p](y) - I_2[\hat{F}_p](y)
  - I_3[\hat{F}_p](y) + I_4[\hat{F}_p](y) \right\}.
\ees
From this we compute
\bes
\py \hat{u}_p(y) = \AKp^2 \left\{ -I_1[\hat{F}_p](y) - I_2[\hat{F}_p](y)
  - I_3[\hat{F}_p](y) - I_4[\hat{F}_p](y) \right\} + \hat{F}_p(y),
\ees
which gives
\begin{gather*}
\hat{u}_p(0) = 0,
\quad
\hat{u}_p(-a) = \AKp \left\{ -I_1(-a) - I_2(-a) \right\},
\\
\py \hat{u}_p(0) = -2\AKp^2 I_3(0) + \hat{F}_p(0).
\end{gather*}
Following the proof of Lemma~\ref{lem:F} we can show the following bounds.
\begin{Lemma}
\begin{align*}
& \AKp^m \Norm{\hat u_p}{L^2} \leq C \Norm{\hat{F}_p}{L^2},
  && 0 \leq m \leq 1, & \forall\ p \in \Gamma', \\
& \Norm{\py\hat u_p}{L^2} \leq C \Norm{\hat{F}_p}{L^2},
  && & \forall\ p \in \Gamma', \\
& \Abs{\py\hat u_p(0)} \leq C \AKp^{1/2} \Norm{\hat{F}_p}{L^2},
  && & \forall\ p \in \Gamma', \\
& \AKp^m \Abs{\hat u_p(-a)} \leq C \Norm{\hat{F}_p}{L^2},
  && 0 \leq m \leq 1/2, & \forall\ p \in \Gamma'.
\end{align*}
\end{Lemma}
With this we obtain the following estimates.
\begin{description}
  %
  %
  \item[(i.) Field:] 
  \bes
  \max \left\{ \XNorm{\tu}{s}, \XNorm{\AKD^{1/2} \tu}{s} \right\}
    \leq C \XNorm{\tF^y}{s}.
  \ees
  %
  %
  \item[(ii.) Vertical derivative:] 
  \begin{align*}
  & \XNorm{\py\tu}{s} \leq C \XNorm{\tF^y}{s}, \\
  & \max \left\{ \XNorm{\AKD^{1/2} \py \tu}{s},
  \SobNorm{\py \tu (0)}{s}^2 \right\} \leq C \XNorm{\AKD^{1/2}\tF^y}{s}.
  \end{align*}
  %
  %
  \item[(iii.) Horizontal derivative:] 
  \begin{align*}
  & \XNorm{K^T \Grada{\tu}}{s} \leq C \XNorm{\tF^y}{s}, \\
  & \max \left\{ \XNorm{\AKD^{1/2}K^T \Grada{\tu}}{s},
  \XNorm{K^T \Grada{\tu}(0)}{s} \right\} \leq C \XNorm{\AKD^{1/2}\tF^y}{s}.
  \end{align*}
  %
  %
  \item[(iv.) DNO at the lower boundary:]
  \bes
  \SobNorm{\tT[ \tu(-a) ]}{s} 
    = \SobNorm{\AKD\tu(-a)}{s} \leq C \XNorm{\AKD^{1/2}\tF^y}{s}.
  \ees
  \end{description}
\end{description}
\end{proof}

%
%

\section{Proof of Lemma~\ref{lem:F}}
\label{App:InhomogPDE}

We now provide the full proof of Lemma~\ref{lem:F}.
\begin{proof}
The proof for the case when $p = 0$ is straightforward and we focus 
on the case when $p \neq 0$. To simplify our developments we write
\bes
\hat{u}_p(y) = I_1[\hat{F}_p](y) - I_2[\hat{F}_p](y) 
    + I_3[\hat{F}_p](y) - I_4[\hat{F}_p](y),
\ees
where
\begin{align*}
I_1[\hat{F}_p](y) & := \frac{\exp(\AKp y)}{2 \AKp}
  \int_y^0 \exp(\AKp t) \hat{F}_p(t) \dt, \\
I_2[\hat{F}_p](y) & := \frac{\exp(\AKp y)}{2 \AKp}
  \int_y^0 \exp(-\AKp t) \hat{F}_p(t) \dt, \\
I_3[\hat{F}_p](y) & := \frac{\exp(\AKp y)}{2 \AKp}
  \int_{-a}^y \exp(\AKp t) \hat{F}_p(t) \dt, \\
I_4[\hat{F}_p](y) & := \frac{\exp(-\AKp y)}{2 \AKp}
  \int_{-a}^y \exp(\AKp t) \hat{F}_p(t) \dt.
\end{align*}
From this we obtain
\begin{gather*}
\py \hat{u}_p(y) = \AKp \left\{ I_1[\hat{F}_p](y) - I_2[\hat{F}_p](y) 
  + I_3[\hat{F}_p](y) + I_4[\hat{F}_p](y) \right\},
\\
\hat{u}_p(0) = 0,
\quad 
\hat{u}_p(-a) = I_1[\hat{F}_p](-a) - I_2[\hat{F}_p](-a),
\\
\py \hat{u}_p(0) = 2\AKp \left\{ I_3[\hat{F}_p](0)\right\}.
\end{gather*}

Computing the $L^2$ norm of $I_1[\hat{F}_p](y) - I_2[\hat{F}_p](y)$,
we obtain 
\begin{align*}
\Norm{I_1[\hat{F}_p](y) - I_2[\hat{F}_p](y)}{L^2}^2 
& = 
\int_{-a}^0 \exp(\AKp 2a) \dy \Abs{
  \int_y^0 E_p(y) S_p(t) \hat{F}_p(t) \dt}^2\\
  & =
\frac{1}{\Abs{K^T p}^2}\int_{-a}^0 \exp(2\Abs{K^Tp}y) \; dy
\left(\int_y^0\sinh(\Abs{K^T p}t) \hat{F}_p(t)\; dt\right)^2 \\
& \leq \frac{1}{\Abs{K^T p}^2}\int_{-a}^0 \exp(2\Abs{K^Tp}y) \; dy
\int_y^0 \sinh^2(\Abs{K^T p}t)\;dt
\int_y^0 |\hat{F}_p(t)|^2 \; dt \\
&\leq \frac{a}{\Abs{K^T p}^2} 
\Norm{\hat{F}_p}{L^2}^2
\int_{-a}^0 \exp(2\Abs{K^Tp}y) \sinh^2(\Abs{K^T p}y) \; dy,
\end{align*}
since
\bes
\int_y^0 \sinh^2(\AKp t) \dt
  \leq \sinh^2(\AKp y) \int_y^0 \dt
  \leq a \sinh^2(\AKp y).
\ees
Continuing
\bes
\Norm{I_1[\hat{F}_p](y) - I_2[\hat{F}_p](y)}{L^2}^2 
  \leq a f\left(\AKp\right)
  \Norm{\hat{F}_p}{L^2}^2 \\
  \leq C \Norm{\hat{F}_p}{L^2}^2,
\ees
where
\begin{align*}
f\left(\AKp\right) &= -
  \frac{1}{16\AKp^3}\left(\exp(-4\Abs{K^T p}a)\big)
  - 4\exp(-2\Abs{K^T p}a) - 4a\AKp + 3\right) \\
&\to 
\begin{cases}
\frac{a^3}{3}, & \AKp \to 0, \\
\frac{a}{4 \AKp^2}, & \AKp \to \infty.
\end{cases}
\end{align*}
Similarly, one can prove that
\bes 
\Norm{I_3[\hat{F}_p](y) - I_4[\hat{F}_p](y)}{L^2}^2 
  \leq C \Norm{\hat{F}_p}{L^2}^2,
\ees
thus, we have
\be
\label{eq:up:L2:F}
\Norm{\hat u_p}{L^2}
  \leq C \Norm{\hat{F}_p}{L^2}.
\ee

Next, we are going to show that
\bes
\AKp^2 \Norm{\hat u_p}{L^2} \leq C \Norm{\hat{F}_p}{L^2}.
\ees
Computing the squared $L^2$ norm of 
$\AKp^2 I_1[\hat F_p]$ we have
\begin{align*}
\AKp^4\Norm{I_1[\hat F_p]}{L^2}^2 & =
\frac{\Abs{K^T p}^2}{4}\int_{-a}^0 \exp(2\Abs{K^Tp}y) \dy
\Abs{\int_y^0 \exp(\Abs{K^T p}t)\hat{F}_p(t)\dt}^2 \\
& \leq \frac{\Abs{K^T p}^2}{4} \int_{-a}^0 \exp(2\Abs{K^Tp}y)\dy
\int_y^0 \exp(2\Abs{K^T p}t)\dt
\int_y^0 |\hat{F}_p(t)|^2 \dt \\
&\leq \frac{\Abs{K^T p}^2}{4} 
\Norm{\hat{F}_p}{L^2}^2
\int_{-a}^0 \exp(2\Abs{K^Tp}y) \dy
\int_{-a}^0 \exp(2\Abs{K^T p}t)\dt \\
&\leq \frac{1}{16} \Norm{\hat{F}_p}{L^2}^2
\left(1 - \exp(-2\Abs{K^Tp}a)\right)^2 \\
&\leq C \Norm{\hat{F}_p}{L^2}^2.
\end{align*}
An estimation of $L^2$ norm of $\AKp^2 I_2[\hat F_p]$
proceeds as follows
\begin{align*}
\AKp^2\Norm{I_2[\hat F_p]}{L^2} & = 
\frac{\AKp}{2} \Big(
\int_{-a}^0 \dy 
\Abs{\int_y^0 \exp\left(\Abs{K^Tp}(y-t)\right)\hat{F}_p(t)\dt}^2
\Big)^{1/2} \\
& \leq \frac{\AKp}{2} \left(
\int_{-a}^0 \dy 
\left(\int_{-a}^0 \exp\left(\Abs{K^Tp}(y-t)\right)\Abs{\hat{F}_p(t)}\dt\right)^2
\right)^{1/2} \\
& \leq \frac{\AKp}{2}
\int_{-a}^0 \exp(\Abs{K^Tp}y) \dy 
\Big(
 \int_{-a}^0 \Abs{\hat{F}_p(y)}^2 \dy
\Big)^{1/2}\\
& = \frac{1}{2}(1 - \exp(-\AKp a))\Norm{\hat{F}_p}{L^2} \\
& \leq C \Norm{\hat{F}_p}{L^2}.
\end{align*}
Similarly, we can show that
\bes
\AKp^2\Norm{I_j[\hat F_p]}{L^2} \leq C \Norm{\hat{F}_p}{L^2},
\quad j = 3,4.
\ees
Therefore, we completed the proof of 
\be
\label{eq:up:AKp2L2:F}
\AKp^2\Norm{\hat u_p}{L^2} \leq C \Norm{\hat{F}_p}{L^2}.
\ee
Interpolating \eqref{eq:up:L2:F} and \eqref{eq:up:AKp2L2:F}, we obtain
the first inequality in Lemma \ref{lem:F}. The second inequality in 
Lemma \ref{lem:F} can be proved similarly.

At $y = 0$, we have
\begin{align*}
\Abs{\py \hat u_p(0)} 
&= \Abs{\int_{-a}^0 \exp(\Abs{K^Tp}t)\hat F_p(t)\dt}\\
&\leq \Big(\int_{-a}^0 \exp(2\Abs{K^Tp}t)\dt\Big)^{1/2} \Norm{\hat F_p}{L^2}\\
& = \left(\frac{1 - \exp(-2\Abs{K^T p}a)}{2\Abs{K^Tp}}\right)^{1/2}
  \Norm{\hat F_p}{L^2}.
\end{align*}
Therefore, we obtain
\bes
\Abs{\py \hat u_p(0)} \leq C \Norm{\hat F_p}{L^2}, \qquad
\AKp^{1/2}\Abs{\py \hat u_p(0)} \leq C \Norm{\hat F_p}{L^2}.
\ees
Through interpolation, we finish the proof of the third inequality in 
Lemma~\ref{lem:F}. We use the following two inequalities to prove the last inequality in Lemma \ref{lem:F}. At $y = -a$, we have
\begin{align*}
\Abs{\hat{u}_p(-a)} 
  &\leq \frac{\exp(-\Abs{K^Tp}a)}{\Abs{K^T p}}
    \int_{-a}^0\Abs{\sinh(\Abs{K^Tp}t)\hat{F}_p(t)} \dt \\
  &\leq \frac{\exp(-\Abs{K^Tp}a)}{\Abs{K^T p}} 
    \int_{-a}^0 \sinh(\Abs{K^Tp}a) \Abs{\hat{F}_p(t)} \dt \\
  &\leq \frac{1 - \exp(-2\Abs{K^Tp}a)}{2 \Abs{K^Tp}}\int_{-a}^0 |\hat{F}_p(t)|\dt,
\end{align*}
since $\Abs{\sinh(\AKp t)} \leq \sinh(\AKp a)$ on $[-a, 0]$.
Continuing,
\bes
\Abs{\hat{u}_p(-a)} 
  \leq \frac{1 - \exp(-2\AKp a)}{2 \AKp}
    \sqrt{a} \Norm{\hat{F}_p}{L^2}
  \leq C \Norm{\hat{F}_p}{L^2}.
\ees
Finally, since $\sinh(z) \leq \exp(-z)$ for $z < 0$,
\begin{align*}
\AKp^{3/2} \Abs{\hat{u}_p(-a)} 
  & \leq \AKp^{1/2}\exp(-\Abs{K^Tp}a)
  \int_{-a}^0\exp(-\Abs{K^Tp}t)\Abs{\hat{F}_p(t)} \dt \\
  & \leq \AKp^{1/2}\exp(-\Abs{K^Tp}a) 
  \left(\int_{-a}^0 \exp(-2\Abs{K^Tp}t) \dt\right)^{1/2}\Norm{\hat{F}_p}{L^2} \\
  & \leq C \Norm{\hat{F}_p}{L^2}.
\end{align*}
\end{proof}

%
%

\bibliography{nicholls}

\end{document}